\documentclass[a4paper, 11pt, english]{article}
\usepackage{xcolor}
\usepackage{preambel}
\usepackage[toc]{appendix}
\usepackage{a4wide}
\usepackage{hyperref}
\usepackage{authblk}

\usepackage[url=false,eprint=false]{biblatex} 
\begin{document}
	\title{\textbf{\large Local spectral density and Slepian concentration for spherical Fourier-Bessel truncation spaces}}
	
	\author[1]{X. Huang\thanks{xinpeng.huang@csu.edu.cn}}
	
	\affil[1]{
		Central South University, Lushannan Road 932, Changsha, China
	}
	\date{}
	\maketitle
	
	\begin{abst}
		We study diagonal kernel asymptotics and concentration spectra for a family of non-translation-invariant spectral projections in $\R^d$, $d\geq 2$. The projections are obtained from the classical Paley--Wiener projection by imposing, in spherical coordinates, an additional cutoff in the spherical-harmonic degree. Equivalently, they are the spherical Fourier--Bessel (SFB) truncation spaces in which, in addition to the radial Hankel/Bessel bandwidth $K$, only spherical harmonic degrees $\degaInCFive\leq\degAInCFive$ are retained. This angular cutoff preserves rotation invariance but breaks translation invariance, so the diagonal reproducing kernel has a spatially varying radial profile. In the coupled asymptotic regime $\degAInCFive/K\to\kappa$, we identify the limiting profile of the normalized diagonal reproducing kernel $K^{-d}\KLK(x,x)$, interpreted as the local density encoded by the SFB projection. The profile is the rescaled radial transition function ${\functionInCFiveFB}^{\langle d\rangle}_{\langle\kappa\rangle}(\|x\|)={\functionInCFiveFB}^{\langle d\rangle}(\|x\|/\kappa)$. Its constant plateau recovers the constant density of the classical Paley--Wiener projection for $\|x\|\lesssim\kappa$, while its far-field tail, when weighted by the spherical volume element, yields a Hankel-type radial density law with angular-bandwidth factor $\kappa^{d-1}$. Thus the classical Paley--Wiener concentration problem is recovered at the endpoint $\kappa=\infty$, whereas finite $\kappa$ exhibits a transition from a Fourier-like local-density region to a Hankel-type radial-density regime, with $\kappa$ setting the radial scale of this transition. Using this local-density asymptotic, we prove an asymptotically bimodal eigenvalue distribution and a Shannon-number formula whose leading coefficient is the integral of this $\kappa$-dependent density over the localization domain.
	\end{abst}
	
	\begin{key}
			Slepian concentration problem, spherical Fourier-Bessel functions, Paley-Wiener spaces, diagonal reproducing kernels, Shannon number, Hankel transform
	\end{key}
	\begin{amsclass} 
		 42C10, 33C10, 47B06, 46E22
	\end{amsclass} 
	\section{Introduction}
	The spherical Fourier--Bessel system (SFB, more precisely the spherical harmonic--Bessel system) provides a generalized orthogonal decomposition on $\R^d$, viewed in spherical coordinates as $\R_{+}\times\Sphered$. Its modes are products of Bessel functions in the radial variable and spherical harmonics in the angular variable. Since these functions arise from separation of variables for the Helmholtz equation and decouple radial and angular behavior, they are natural tools for signal analysis in ball-like geometries. Applications of 3-D SFB functions in astrophysics and cosmology are discussed, for example, in \cite{Heavens1995,SuperFaB,Pratten2013,PhysRevD.90.063515,Fisher_1995-te,PhysRevD.104.103523,Lanusse2015,Rassat2012,6280687,10.1093/mnras/stv2724,Leistedt2012}. Motivated by incomplete observations and spatial feature detection, localized spatial-spectral constructions based on SFB decompositions have also been studied; see, for example, the 3-D SFB wavelet systems in \cite{Lanusse2012,PhysRevD.90.103532}. Most closely related to the present paper, Slepian functions for SFB band-limited spaces in $\R^3$ were introduced in \cite{Khalid2016a}.

	Beyond this specific representation-theoretic origin, the resulting spaces provide a model problem in spectral analysis: they arise from the Paley--Wiener projection by imposing a symmetry-preserving but non-translation-invariant cutoff, and therefore allow one to quantify how local density laws and concentration spectra change when translation invariance is lost.

	The Slepian concentration problem, initiated in the classical works of Slepian, Landau, and Pollak (e.g., \cite{Slepian1961,Slepian1978,Landau1961,Landau1962}), asks for band-limited functions with optimal concentration on a prescribed spatial domain. Equivalently, one studies the eigenvalue decomposition of a spatiospectral concentration operator built from spatial and spectral projections. In the classical Paley-Wiener setting, the resulting Shannon number has the familiar space-bandwidth form. Related constructions for spherical harmonics and Paley-Wiener spaces have become standard tools in communication theory, signal processing, geophysics, and biophysics; see, e.g., \cite{sneeuw99,Thomson82,Simons2006,Simons2006b}. Many variants of the Slepian framework have also been developed, including versions on different geometries \cite{Khalid2016a,leweke18,roddy2023slepian,GraphSlepian17,Miranian_2004,Erb2019ShapesOU}, polynomial analogues \cite{Perlstadt1986,CASTRO2015328,CASTRO2024101600,grunbaum2015time,Grünbaum_2017,GRUNBAUM1983491}, vector and tensor variants \cite{Plattner2014,Michel2022}, and formulations involving multipliers or forward operators instead of sharp spatial/spectral projections \cite{michelsimons18,Erb2013,Erb2015,Plattner2017}.

	\paragraph{Relation to GPSFs and Fourier concentration operators.} The present problem is closely connected with the classical generalized prolate spheroidal function (GPSF) problem. In the GPSF setting (see, e.g., \cite{Slepian1964GPSF,Lederman2017GPSF,Greengard2024GPSF}), the spectral projection is the Euclidean Fourier projection onto a ball in frequency space, and the corresponding Paley--Wiener space is translation-invariant. Recent work has also studied eigenvalue distributions for higher-dimensional Fourier spatio-spectral limiting operators with general spatial and frequency domains; see, for example, \cite{IsraelMayeli2024,HughesIsraelMayeli2025}, as well as related non-asymptotic eigenvalue estimates in \cite{MarcecaRomeroSpeckbacher2024}. Here we consider the proper subspace $\spaceInCFive_{\degAInCFive,K}\subset \mathrm{PW}_{K}$ obtained by retaining, besides the radial Fourier/Bessel bandwidth $K$, only spherical harmonic degrees $n\leq \degAInCFive$. Thus the SFB truncation is not merely a spherical-coordinate representation of the classical Paley--Wiener/GPSF projection; it is a different, non-translation-invariant spectral projection obtained by imposing an additional angular cutoff. In particular, this radial-angular truncation in the spectral domain preserves rotations but loses translations, and consequently the diagonal reproducing kernel $\KLK(x,x)$ becomes a non-constant radial density.

	Thus SFB truncation spaces form a two-parameter refinement of the classical Paley--Wiener space. As $\degAInCFive\to\infty$, they recover $\mathrm{PW}_K$; for finite $\degAInCFive$, the angular cutoff replaces the constant diagonal density of the Paley--Wiener projection by a radial profile.
	
	To identify this profile, we organize the large-bandwidth limit by the relative growth of the two cutoffs. We study the coupled regime
	\[
		\frac{\degAInCFive}{K}\to \kappa\in(0,\infty),
	\]
	where $\kappa$ measures the angular resolution retained per unit radial bandwidth. Equivalently, $\kappa$ sets the radial scale on which the SFB projection behaves like the classical Paley--Wiener projection before entering the Hankel-type far-field regime. The limiting density has the form (see Definition~\ref{def:UandW} for the explicit expression)
	\[
		{\functionInCFiveFB}^{\langle d\rangle}_{\langle\kappa\rangle}(\|x\|)={\functionInCFiveFB}^{\langle d\rangle}(\|x\|/\kappa).
	\]
	This dilation determines which part of the profile is sampled by a fixed spatial domain. Since ${\functionInCFiveFB}^{\langle d\rangle}$ is constant on $[0,1]$, the region $\|x\|\leq\kappa$ has the same constant local density as in the classical Paley--Wiener setting. Corollary~\ref{cor:tailWHankel} shows that, for $\|x\|\gg\kappa$, the density has a Hankel-type radial tail. After multiplication by the spherical volume element, this tail gives a leading radial density of order $\kappa^{d-1}\diffsymbol\|x\|$, matching the radial-length scaling that appears in Hankel concentration at the level of the density law \cite{Abreu2012}. Thus the profile ${\functionInCFiveFB}^{\langle d\rangle}$ quantifies the transition from a $d$-dimensional Paley--Wiener local-density regime near $\|x\|\leq \kappa$ to a Hankel-type radial-density regime for $\|x\|\gg\kappa$. At the endpoint $\kappa\to\infty$, the constant part of the profile expands to all bounded spatial regions, and the constant density of the classical Paley--Wiener projection is recovered.

	The main analytical contributions are as follows.
	\begin{itemize}[$\circ$]
			\item We derive the limiting profile of the normalized diagonal reproducing kernel $K^{-d}\KLK(x,x)$ in the coupled regime $\degAInCFive/K\to\kappa$. This identifies the local spectral density associated with a non-translation-invariant angular cutoff.
			\item We show that this profile interpolates between two density laws: a constant Paley--Wiener plateau and a Hankel-type far-field radial-density law with angular-bandwidth factor $\kappa^{d-1}$.
			\item We prove an asymptotically bimodal eigenvalue distribution and a Shannon-number law for the associated concentration operators, with leading coefficient given by the integral of the limiting local density over the localization domain.
	\end{itemize}

	Thus $K^d$ provides the usual bandwidth scaling, while the angular cutoff replaces the translation-invariant phase-space coefficient by a radial, domain-dependent density. In particular, domains of equal volume may have different leading Shannon numbers when they occupy different radial regions.

	More precisely, for a bounded Lipschitz domain $D$, we study the eigenvalues of the SFB-based spatiospectral concentration operator. These eigenvalues lie in $[0,1]$ and measure the spatial concentration of SFB band-limited functions in $D$. We show that the number of eigenvalues near one, normalized by $K^d$, converges to the integral of the limiting density over $D$, whereas the number of eigenvalues in any fixed intermediate interval is $\smallo(K^d)$. This gives the Shannon-number asymptotic its interpretation as the leading effective dimension of the concentration problem.

	We also demonstrate the eigenvalue-distribution theorem numerically for radial localization domains of disk and annular type. These experiments are intended only to illustrate the asymptotic eigenvalue distribution and the radial dependence of the Shannon coefficient; they are not used in the proofs of the main results. The experiments compare the observed eigenvalue plunge with the theoretical leading count and confirm that equal-volume domains can have different effective dimensions depending on their radial placement.

	The paper is organized as follows. Section \ref{sec:preliminary} recalls the SFB band-limited spaces, their reproducing kernels, and the Bessel-function estimates used later. Section \ref{secC5:sec2} proves the asymptotic formulas for the diagonal reproducing kernel. Section \ref{secC5:SlepianProblem} applies these formulas to the eigenvalue distribution of the Slepian concentration operator. Section~\ref{secC5:radial-nystrom} develops the radial block discretization and numerically illustrates the dependence of the eigenvalue plunge on radial placement. Appendix~\ref{secC5:numericalremarks} records additional bounded-ball analogues and supplementary kernel observations that are not used in the proofs of the main results.

	\paragraph{Notation} We use $\mathbb{R}$, $\mathbb{R}_{+}$, $\mathbb{R}^d$, $\mathbb{N}$, and $\mathbb{N}_{0}$ for the real numbers, positive real numbers, $d$-dimensional real vectors, positive integers, and non-negative integers, respectively. The closed $d$-dimensional ball of radius $r>0$ is denoted by $\BBCFour{r}=\{x\in\R^d:\|x\|\leq r\}$, and $\BB^d=\BBCFour{1}$ denotes the closed unit ball. The unit sphere is $\Sphered=\partial \B^d=\{x\in\R^d:\|x\|=1\}$. For a linear vector space $X$, $\mathrm{dim}(X)$ denotes its dimension and $\mathrm{proj}_{X}$ denotes the orthogonal projection onto $X$ from an ambient Hilbert space. For a set $A$, $\sharp A$ denotes its cardinality and $\chi_A$ its characteristic function.

\section{Preliminaries}\label{sec:preliminary}
\subsection{Helmholtz equation in spherical coordinates}
The spherical Fourier--Bessel basis arises from solving the free-space Helmholtz equation (i.e., finding eigenfunctions of the Laplacian)
\begin{align}\label{eqnC5:Helmholtz equation}
	(	\laplacian +k^2)f=0,
\end{align}
in $\R^d$, with $d\geq 2$. In spherical coordinates,  
\begin{align}
	\laplacian f&= \frac{1}{r^{d-1}}\frac{\partial}{\partial r}\bigg( r^{d-1}\frac{\partial}{\partial r}\bigg)f+\frac{1}{r^2}\laplacian_{\Sphere} f\\ \nonumber
	&=\frac{\partial^2}{\partial r^2}f+\frac{d-1}{r}\frac{\partial}{\partial r}f+\frac{1}{r^2}\laplacian_{\Sphere} f.
\end{align}
Here $\laplacian_{\Sphere}$ denotes the Laplace--Beltrami operator on $\Sphered$. Separating variables by writing $f(x)=R(r)\Xi(\xi)$, with $r=\|x\|$ and $\xi=x/\|x\|\in\Sphered$, yields the following equations for $R(r)$ and $\Xi(\xi)$:
\begin{subnumcases}{}
	(\laplacian_{\Sphere} +{\tilde{\degaInCFive}}^2)\Xi(\xi)=0, \label{eqnC5:sp-helmholtz}\\
	\left[\frac{\diffsymbol^2}{\diffsymbol r^2}+\frac{d-1}{r}\frac{\diffsymbol}{\diffsymbol r}+(k^2-\frac{{\tilde{\degaInCFive}}^2}{r^2})\right]R(r)=0. \label{eqnC5:r-helmholtz}
\end{subnumcases}

The spherical equation~\eqref{eqnC5:sp-helmholtz} has solutions only for the discrete eigenvalues of $-\laplacian_{\Sphere}$, namely ${\tilde{\degaInCFive}}^2=\degaInCFive(\degaInCFive+d-2)$ with $\degaInCFive\in\mathbb{N}_0$. The corresponding solutions form the $\left[\binom{\degaInCFive+d-1}{\degaInCFive}-\binom{\degaInCFive+d-3}{\degaInCFive-2}\right]$-dimensional space of spherical harmonics, denoted by $H_{\degaInCFive}^{d}$. We choose an arbitrary orthonormal basis $\{Y_{\degaInCFive,\degb}:\degb=1,\ldots,\mathrm{dim}(H_{\degaInCFive}^d)\}$ of this space. We also write $\mathrm{Harm}_{\degAInCFive}(\Sphered)$ for the orthogonal sum of $H_{\degaInCFive}^d$ over all $\degaInCFive\leq\degAInCFive$, that is, $\mathrm{Harm}_{\degAInCFive}(\Sphered):=\oplus_{0\leq\degaInCFive\leq\degAInCFive}H_{\degaInCFive}^{d}$.

For the radial equation in~\eqref{eqnC5:r-helmholtz}, Lommel's transformation of Bessel's equation \cite[Chap. 4.31]{Watson}, with ${\tilde{\degaInCFive}}^2=\degaInCFive(\degaInCFive+d-2)$, gives solutions of the form
\begin{align*}
	R(r)={C} r^{\frac{2-d}{2}}\mathcal{E}_{\degaInCFive+\frac{d-2}{2}}(kr),\quad {C}\in \R,
\end{align*}
where $\mathcal{E}_{v}$ denotes the Bessel function of order $v$. Imposing well-definedness of $f$ at the origin requires $\mathcal{E}_v$ to be the Bessel function \emph{of the first kind}, denoted by $J_v$ in the standard notation.

\subsection{Spherical Fourier-Bessel band-limited function space}

We define the SFB truncation space through its angular decomposition and its Hankel spectral support. This avoids treating fixed-frequency Bessel waves, which are generalized eigenfunctions rather than elements of $L^2(\R^d)$, as an ordinary Hilbert-space basis.

Let $f\in L^2(\R^d)$. In spherical coordinates it has the orthogonal expansion
\begin{align}\label{eqnC5:angular-decomposition-definition}
f(r\xi)
&=\sum_{\degaInCFive=0}^{\infty}
\sum_{\degb=1}^{\dim(H_{\degaInCFive}^{d})}
f_{\degaInCFive,\degb}(r)Y_{\degaInCFive,\degb}(\xi),\\
f_{\degaInCFive,\degb}(r)
&=\int_{\Sphered}f(r\xi)
\overline{Y_{\degaInCFive,\degb}(\xi)}
\diffsymbol\omega(\xi),
\end{align}
where the expansion is understood in $L^2(\R^d)$. Here and below, for each $\degaInCFive$, the index $\degb$ ranges from $1$ to $\dim(H_{\degaInCFive}^{d})$. Put
\[
\nu_{\degaInCFive}:=\degaInCFive+\frac{d-2}{2},
\qquad
g_{\degaInCFive,\degb}(r):=
r^{\frac{d-2}{2}}f_{\degaInCFive,\degb}(r).
\]
Then $g_{\degaInCFive,\degb}\in L^2(\R_+,r\diffsymbol r)$. For $\alpha\geq-\nicefrac{1}{2}$, let $\mathcal H_\alpha$ denote the unitary Hankel transform on $L^2(\R_+,r\diffsymbol r)$,
\[
(\mathcal H_\alpha g)(k)
:=\int_0^\infty g(r)J_\alpha(kr)r\diffsymbol r,
\]
where the integral is interpreted through the standard $L^2$ extension. For $\degAInCFive\in\N_0$ and $K>0$, we define
\begin{equation}\label{eqnC5:strict-definition-space}
\begin{aligned}
\spaceInCFive_{\degAInCFive,K}
:=\big\{f\in L^2(\R^d):\;&
f_{\degaInCFive,\degb}=0
\text{ for }\degaInCFive>\degAInCFive,\ \text{and}\\[-0.2em]
&
\operatorname{ess\,supp}
\mathcal H_{\nu_{\degaInCFive}}
g_{\degaInCFive,\degb}
\subset[0,K]
\text{ for }0\leq\degaInCFive\leq\degAInCFive
\big\}.
\end{aligned}
\end{equation}
We call $\degAInCFive$ the spherical harmonic bandwidth and $K$ the radial Hankel, or Bessel, bandwidth.

In the same generalized spectral sense in which a Paley--Wiener space is associated with plane waves of frequencies $\|\omega\|\leq K$, the space $\spaceInCFive_{\degAInCFive,K}$ may be viewed as retaining the SFB modes
\[
r^{\frac{2-d}{2}}
J_{\degaInCFive+\frac{d-2}{2}}(kr)
Y_{\degaInCFive,\degb}(\xi),
\qquad
0\leq\degaInCFive\leq\degAInCFive,\quad 0\leq k\leq K.
\]

Equivalently, this means imposing the angular cutoff $\degaInCFive\leq\degAInCFive$ together with the Hankel spectral support condition $k\in[0,K]$ in each angular channel, as in~\eqref{eqnC5:strict-definition-space}.

For comparison, with the Fourier transform, the classical Paley--Wiener space is
\begin{align}\label{eqnC5:PW-strict-definition}
\mathrm{PW}_{K}
:=\{f\in L^2(\R^d):
\operatorname{ess\,supp}\mathcal Ff
\subset\overline{\BBCFour{K}}\}.
\end{align}
The Fourier transform in spherical coordinates decomposes into the same angular channels and the Hankel transforms of orders $\nu_{\degaInCFive}$. Consequently,
\begin{align}\label{eqnC5:SFB-as-angular-truncation}
\spaceInCFive_{\degAInCFive,K}
=\big\{f\in\mathrm{PW}_K:
f_{\degaInCFive,\degb}=0,\quad
\forall\degaInCFive>\degAInCFive\big\}.
\end{align}
Thus $\spaceInCFive_{\degAInCFive,K}$ is precisely the Paley--Wiener space with an additional angular cutoff. As $\degAInCFive\to\infty$, the corresponding orthogonal projections converge strongly to the projection onto $\mathrm{PW}_K$.


This truncation of the spherical harmonic coefficients is natural in several applications. For example, in planetary-scale geophysical inverse problems, observational data are often collected on
a spherical orbit above the planet. The resulting gravity or magnetic-field models are typically represented by finitely many spherical harmonics. Consequently, any term in the SFB expansion of an interior source whose angular degree exceeds the maximal spherical harmonic degree of the data is invisible to the corresponding model. From a mathematical perspective, the space is also relevant to physical systems in which $k$ (radial wavenumber) and $\degaInCFive$ (angular momentum) are coupled.

The analysis in this paper is based on the fact that $\spaceInCFive_{\degAInCFive, K}$ is a reproducing kernel Hilbert space. We therefore study the reproducing kernel $\KLK(x,y)$ of $\spaceInCFive_{\degAInCFive, K}$, which is characterized by
\begin{align}\label{eqnC5:defOfRPK}
	\mathrm{proj}_{\spaceInCFive_{\degAInCFive, K}}f(x)=\int_{\R^d}f(y)\overline{\KLK(x,y)}\diffsymbol y, \quad \forall f\in L^2(\R^d).
\end{align}
Our main objective is to describe the diagonal $\KLK(x,x)$ for large $\degAInCFive$ and $K$. In~\eqref{eqnC5:defOfRPK}, $\KLK(x,\cdot)$ and $\KLK(\cdot,y)$ are interpreted as $L^2$ functions. When writing $\KLK(x,x)$, however, we refer to the unique continuous representative $\KLK\in C(\R^d\times\R^d)$ satisfying~\eqref{eqnC5:defOfRPK}, so the diagonal value is unambiguous. By the reproducing property, it also satisfies
\begin{align}
	\KLK(x,x)=\int_{\R^d} \KLK(x,y) \overline{\KLK(x,y)} \diffsymbol y =\|\KLK(x,\cdot)\|^2_{L^2(\R^d)}.
\end{align} 
This function is crucial in the Slepian concentration problem for SFB systems, especially in the characterization of the Shannon number discussed in Section \ref{secC5:SlepianProblem}.

The spaces $\spaceInCFive_{\degAInCFive, K}$ and $\mathrm{PW}_{K}$ have different invariance properties. The Paley--Wiener space $\mathrm{PW}_{K}$ is invariant under translations and rotations, whereas $\spaceInCFive_{\degAInCFive, K}$ retains only rotation invariance. This distinction is reflected in their reproducing kernels: for any $x,y\in\R^d$ and any rotation matrix $\Ro$,
\begin{align*}
	\mathcal{K}_{\mathrm{PW}_{K}}(x,y)=\mathcal{K}_{\mathrm{PW}_{K}}(0,\Ro(x-y)) && \KLK(x,y)=\KLK(\Ro x,\Ro y),
\end{align*}
where $\mathcal{K}_{\mathrm{PW}_{K}}$ denotes the reproducing kernel of $\mathrm{PW}_{K}$. In particular, for the space $\mathrm{PW}_{K}$, the diagonal reproducing kernel is a constant function, since $\mathcal{K}_{\mathrm{PW}_{K}}(x,x)=\mathcal{K}_{\mathrm{PW}_{K}}(0,0)$ for any $x\in\R^d$. For $\spaceInCFive_{\degAInCFive, K}$, the diagonal kernel $\KLK(x,x)$ is radial, i.e., it depends on $\|x\|$. As shown in Section \ref{secC5:sec2}, it exhibits a nontrivial spatial pattern.

\paragraph{Formulation of the reproducing kernel:}

With the angular coefficients from~\eqref{eqnC5:angular-decomposition-definition}, the orthogonal projection onto $\spaceInCFive_{\degAInCFive,K}$ acts independently in each retained angular channel:
\begin{align}
	\mathrm{proj}_{\spaceInCFive_{\degAInCFive, K}}f(r\xi)=\sum_{\degaInCFive=0}^{\degAInCFive}\sum_{\degb=1}^{ \mathrm{dim}(H_{\degaInCFive}^{d})} \mathrm{proj}_{d,\degaInCFive,K} f_{\degaInCFive,\degb} (r) Y_{\degaInCFive,\degb}(\xi),
\end{align}
where $\mathrm{proj}_{d,\degaInCFive,K}$ is the weighted Hankel band-limiting projection in $L^2(\R_+,r^{d-1}\diffsymbol r)$. Explicitly, let
\begin{align} \label{eqnC5:WeightedHankelProjection}
	\tilde{\opFont{P}} := & \mathcal{M}_{r^{\frac{2-d}{2}}}\circ \mathrm{proj}_{\mathcal{H}_{\degaInCFive+\frac{d-2}{2}},K}\circ \mathcal{M}_{r^{\frac{d-2}{2}}}: \\ \nonumber
	& L^2(\R_{+}, r^{d-1}\diffsymbol r)\to L^2(\R_{+}, r\diffsymbol r)\to L^2(\R_{+}, r\diffsymbol r)\to L^2(\R_{+}, r^{d-1}\diffsymbol r),
\end{align}
where $\mathcal{M}_g$ denotes the multiplication operator $f\to fg$ and $\mathrm{proj}_{\mathcal{H}_{\degaInCFive+\frac{d-2}{2}},K}$ is the standard Hankel band-limited projection with order $\degaInCFive+\frac{d-2}{2}$ and bandlimit $K$, i.e.,
\begin{align}
	\mathrm{proj}_{\mathcal{H}_{\degaInCFive+\frac{d-2}{2}},K} f= & \mathcal{H}_{\degaInCFive+\frac{d-2}{2}}^{-1} \left[\chi_{[0,K]}\cdot \mathcal{H}_{\degaInCFive+\frac{d-2}{2}} f\right] : && L^2(\R_{+}, r\diffsymbol r)\to L^2(\R_{+}, r\diffsymbol r).
\end{align} 

Here $\mathcal H_\alpha$ is the Hankel transform introduced above (see also \cite[Chap. 8]{Titchmarsh:1986:ITF}). Its unitarity shows directly that $\tilde{\opFont P}$ is the orthogonal projection determined by the support condition in~\eqref{eqnC5:strict-definition-space}; hence $\tilde{\opFont P}=\mathrm{proj}_{d,\degaInCFive,K}$.

The integral representation of the Hankel band-limiting projection is 
\begin{align}\label{eqnC5:IntegralHankelProjection}
	\mathrm{proj}_{\mathcal{H}_{\alpha},K} f(r)=\int_{0}^{\infty} f(r') \left[\int_{0}^{K} J_{\alpha}(kr)J_{\alpha}(kr') k \diffsymbol k\right] r'\diffsymbol r',
\end{align}
inserting~\eqref{eqnC5:IntegralHankelProjection} into~\eqref{eqnC5:WeightedHankelProjection} gives
\begin{align}\label{eqnC5:IntegralWeightedHankelProjection}
	\mathrm{proj}_{d,\degaInCFive,K} f(r)= r^{\frac{2-d}{2}}\int_{0}^{\infty} {r'}^{\frac{d-2}{2}}f(r') \left[\int_{0}^{K} J_{\degaInCFive+\frac{d-2}{2}}(kr)J_{\degaInCFive+\frac{d-2}{2}}(kr') k \diffsymbol k\right] r'\diffsymbol r'.
\end{align}
Therefore,
\begin{align}\label{eqnC5:IntegralProjectionPi}
	\mathrm{proj}_{\spaceInCFive_{\degAInCFive, K}}f(r\xi)
	&=\sum_{\degaInCFive=0}^{\degAInCFive}\sum_{\degb=1}^{ \mathrm{dim}(H_{\degaInCFive}^{d})}\int_{\Sphered}\int_{0}^{\infty} f(r'\xi') \left[(rr')^{\frac{2-d}{2}}\int_{0}^{K} J_{\degaInCFive+\frac{d-2}{2}}(kr)J_{\degaInCFive+\frac{d-2}{2}}(kr') k \diffsymbol k\right] \\ \nonumber
	&\quad \quad \quad \quad \times Y_{\degaInCFive,\degb}(\xi) \overline{Y_{\degaInCFive,\degb}(\xi')}\underbrace{[(r')^{d-2}r']\diffsymbol r'\diffsymbol \omega(\xi')}_{=\diffsymbol x'}\\ \nonumber
	&= \int_{\R^d}f(x')	\KLK(x,x') \diffsymbol x',
\end{align}
which yields
\begin{align}\label{eqnC5:K}
	\nonumber \KLK(x,y)&=  \sum_{\degaInCFive=0}^{\degAInCFive}\left\{\left[(r_xr_y)^{\frac{2-d}{2}}\int_{0}^{K} J_{\degaInCFive+\frac{d-2}{2}}(kr_x)J_{\degaInCFive+\frac{d-2}{2}}(kr_y) k \diffsymbol k\right]\sum_{\degb=1}^{ \mathrm{dim}(H_{\degaInCFive}^{d})} Y_{\degaInCFive,\degb}(\xi_x) \overline{Y_{\degaInCFive,\degb}(\xi_y)}\right\}  \\ 
	&=	 (r_xr_y)^{\frac{2-d}{2}}\sum_{\degaInCFive=0}^{\degAInCFive}\left\{\left[\int_{0}^{K} J_{\degaInCFive+\frac{d-2}{2}}(kr_x)J_{\degaInCFive+\frac{d-2}{2}}(kr_y) k \diffsymbol k\right]\frac{\mathrm{dim}(H_{\degaInCFive}^{d})}{\mathrm{vol}(\Sphered)} P_{\degaInCFive}^{(d)}(\langle\xi_x,\xi_y\rangle)\right\} .
\end{align}
The last equality follows from the addition theorem for spherical harmonics (see, e.g., \cite{Mller1997AnalysisOS}). Here $P_{\degaInCFive}^{(d)}$ denotes the degree-$\degaInCFive$ Legendre polynomial in dimension $d$, normalized by $P_{\degaInCFive}^{(d)}(1)=1$. In particular, $P_{\degaInCFive}^{(2)}$ is the Chebyshev polynomial of the first kind, whereas for $d\geq3$ it can be represented as follows (see, e.g., \cite[Eq. (6.4.41)]{FreGut13}):
\begin{align*}
	P_{\degaInCFive}^{(d)}(t)=\frac{C_{\degaInCFive}^{\frac{d-2}{2}}(t)}{C_{\degaInCFive}^{\frac{d-2}{2}}(1)}= \frac{\substigamma(d-2)\substigamma(n+1)}{\substigamma(d-2+n)}C_{\degaInCFive}^{\frac{d-2}{2}}(t),
\end{align*}
where $C_{\degaInCFive}^{\alpha}$ denotes the Gegenbauer
polynomial. Moreover, $\mathrm{vol}(\Sphered)$ is the volume of the $(d-1)$-dimensional sphere, whose explicit value is $\frac{2\pi^{d/2}}{\substigamma(\frac{d}{2})}$.

When we take $x=y$,~\eqref{eqnC5:K} simplifies to
\begin{align}\label{eqnC5:diagonalK}
	\KLK(x,x)&=\frac{(r_x)^{2-d}}{\mathrm{vol}(\Sphered)}\int_{0}^{K}  \sum_{\degaInCFive=0}^{\degAInCFive} \mathrm{dim}(H_{\degaInCFive}^{d}) J^2_{\degaInCFive+\frac{d-2}{2}}(kr_x) k \diffsymbol k.
\end{align}

\subsection{Preliminaries on the Bessel functions }
Our analysis relies on several properties of Bessel functions, which we collect here for later use. Throughout the paper, we only consider Bessel functions of the first kind, $J_v(t)$, with real non-negative order $v$ and argument $t$.

For large arguments, $J_v$ behaves like a damped cosine function. In particular, we have the well-known asymptotic estimate (cf. \cite[Prop. 5.6]{wendland05})
\begin{align}\label{eqnC5:approxJv1}
J_v(t)= \sqrt{\frac{2}{\pi t}}\{\cos(t-\frac{v\pi}{2}-\frac{\pi}{4})\} +\bigo(t^{-3/2}) \quad (t \to \infty).
\end{align}
The term $-\frac{v\pi}{2}$ in the argument of the cosine in~\eqref{eqnC5:approxJv1} is the phase shift of the Bessel function; it represents a delay in the onset of oscillatory behavior (cf. \cite[p.~7]{WRB08}). Indeed, for large $v$, $J_v$ is small on the interval $[0,v]$ in the following sense:
for $v>0$ and $0\leq t\leq v$, both $J_v(t)$ and $J_v'(t)$ (where $J_v'$ denotes the derivative of $J_v$) are increasing functions of $t$, with upper bounds given by
\begin{enumerate}
\item \cite[Chap. 8.54]{Watson} $v^{-1/3}J_v(v)$, as a function of real $v$, is increasing and converges to $\frac{\substigamma(1/3)}{2^{2/3}3^{1/3}\pi}\approx 0.44731$. As a result,
\begin{align}\label{bound:inpBesselj}	
	\max_{0\leq t\leq v}|J_v(t)| = J_v(v)\leq \frac{\substigamma(1/3)}{2^{2/3}3^{1/3}\pi} v^{-1/3}.
\end{align}
\item By \cite[Chap. 8.54]{Watson}, $v^{-2/3}J_{v}^{'}(v)$ is increasing as a function of real $v$ and converges to $\frac{3^{1/6}\substigamma(2/3)}{2^{1/3}\pi}\approx 0.41085$. Consequently,
\begin{align}\label{bound:derivativeBesselj0to1}	
	\max_{0\leq t\leq v}|J_v^{'}(t)|=J^{'}_v(v)  \leq \frac{3^{1/6}\substigamma(2/3)}{2^{1/3}\pi} v^{-2/3}.
\end{align}
\end{enumerate}
Moreover, for $0<t<v$, $J_v(t)$ is bounded by $(\nicefrac{t}{2})^{v}/\substigamma(v+1)$, which gives a sharper estimate near zero. This follows from the product representation
\begin{align}\label{eqnC5:multiplicationJv}
J_v(t)=\frac{t^v}{2^v\substigamma(v+1)} \prod_{n=1}^{\infty} \left(1-\frac{t^2}{j^2_{v,n}}\right).
\end{align}
Here $j_{v,n}$ denotes the $n$th positive zero of $J_v$; all these zeros exceed $v$ when $v>0$.

For $t\in\R_{+}$, $J_v(t)$ attains its maximum near $t=v+\bigo(v^{1/3})$, and the magnitude of this maximum satisfies (see \cite[Eq. (9.5.20)]{abramowitz72} and \cite{Landau_BesselBounds})
\begin{align}\label{bound:maxJv}
\max_{t>0}|J_v(t)|& \simeq 0.674885\cdots v^{-1/3} = \bigo(v^{-1/3}) \quad  (v\to \infty).		\end{align}
For $t\geq sv$, with fixed $s>1$, $|J_v(t)|$ can be bounded using the modulus function $M_v$ (see, e.g., \cite[Chaps. 8.12 and 13.75]{Watson}):
\begin{align}\label{bound:maxJv_Larget}
J_v^2(t)\leq M^2_{v}(t)\lesssim \frac{2}{\pi t}\left(1+\frac{1}{2}\frac{v^2}{(vs)^2}+\frac{1}{2}\frac{3}{4}\frac{v^2\cdot v^2}{(vs)^4}+\cdots\right)\leq \frac{2 {C}_{s}}{\pi t}.		
\end{align}
Here $C_s$ is a constant depending only on $s$, and $M_v^2(t)=J_v^2(t)+Y_v^2(t)$, where $Y_v$ denotes the Bessel function of the second kind.

We will use the following two definite integrals for Bessel functions, which are also known as orthogonality relations. For $v>0$ and $0<b\leq a$,
\begin{align}\label{bound:maxJv2}
\int_{0}^{\infty}\frac{J_v(at)J_v(bt)}{t}\diffsymbol t=\frac{1}{2v}(b/a)^{v}, && \int_{0}^{\infty}J_v(at)J_v(bt)t\diffsymbol t=\frac{\pi}{a^2}\delta(a-b),
\end{align}
where $\delta(a-b)$ denotes the Dirac delta function. The integrability of~\eqref{bound:maxJv2} is guaranteed by~\eqref{eqnC5:approxJv1} and~\eqref{eqnC5:multiplicationJv}.

Finally, we recall the recurrence formulas for Bessel functions (cf. \cite[Chap. 3.2, Eq. (3)-(4)]{Watson}). They are fundamental for treating radial Bessel functions associated with different spherical degrees simultaneously.
\begin{subnumcases}{}
J_{v-1}(t)=	\frac{\diffsymbol}{\diffsymbol t}J_v(t)+\frac{v}{t}J_{v}(t),  \label{eqnC5:recurrenceformula-1}\\
J_{v+1}(t)=	-\frac{\diffsymbol}{\diffsymbol t}J_v(t)+\frac{v}{t}J_{v}(t). \label{eqnC5:recurrenceformula-2}
\end{subnumcases}
\section{Asymptotic behavior of the reproducing kernel along the diagonal}
\label{secC5:sec2}
\subsection{The case $d=2$}
The purpose of this section is to identify the limiting radial profile of the normalized diagonal reproducing kernel and thereby determine the local spectral density of the SFB truncation spaces.
As mentioned above, we are interested in the spatial variation of $\KLK(x,x)$. 
However, rather than deriving an exact formula for $\KLK(x,x)$ in terms of $\degAInCFive$ and $K$, we focus on the asymptotic behavior of $\KLK(x,x)$ as both $\degAInCFive$ and $K$ tend to infinity.
Theorems \ref{thm:characterizeKLKdim2} and \ref{thm:convergenceofshbKtoW} show that, if $\degAInCFive$ is \emph{proportional to} $K$, then the spatial fluctuations of $\KLK(x,x)$ have a stable asymptotic profile, in the sense that $\lim_{\degAInCFive,K\to\infty}\KLK(x,x)/K^d$ exists. The relevant parameter is the limiting ratio $\kappa\in(0,\infty)$, with $\degAInCFive/K\to\kappa$. In this subsection, we begin with the simpler case $d=2$.

\begin{defi}\label{def:expressiong}
Define $\functionInCFiveFA:[0,\infty)\to\R$ by
\begin{align}\label{eqn:defU}
\functionInCFiveFA(r)=
\begin{cases}
1/2 &\quad ( r\leq 1), \\ 
1/2-\frac{1}{\pi}\mathrm{arcsec} (r) &\quad (r>1).
\end{cases}
\end{align}
\end{defi}

\begin{defi}\label{def:dilation}
Let $f$ be a function on $\R_{+}$ and let $\kappa>0$. We denote the dilation of $f$ by
\begin{align*}
f_{\langle \kappa \rangle}(r):=f(\kappa^{-1} r).
\end{align*} 
\end{defi}

The form of $\functionInCFiveFA(r)$ is motivated by the following observation, which plays a central role throughout this paper.
\begin{prop}\label{keyequation}
Let $r>0$ and $v_0\geq0$ be fixed, and let $U$ be the function defined in~\eqref{eqn:defU}. Then
\begin{align}\label{eqnC5:keyequation}
\lim_{ \degAInCFive\to\infty} \sum_{\degaInCFive=0}^{\degAInCFive} J_{v_0+\degaInCFive}^2(\degAInCFive r) = \functionInCFiveFA(r).
\end{align} 
\end{prop}

\begin{proof}

We start from the recurrence formulas for $J_v$. Multiplying both sides of~\eqref{eqnC5:recurrenceformula-2} and~\eqref{eqnC5:recurrenceformula-1} by $J_v(t)$ and subtracting~\eqref{eqnC5:recurrenceformula-2} from~\eqref{eqnC5:recurrenceformula-1}, we obtain
\begin{align}
J_{v-1}(t)J_v(t) - J_{v}(t)J_{v+1}(t)= 2 \frac{\diffsymbol}{\diffsymbol t}J_{v}(t) J_{v}(t)  =\frac{\diffsymbol}{\diffsymbol t}\{J^2_{v}(t)\}.
\end{align}
Taking $v=v_0+\degaInCFive$ for $\degaInCFive=0,1,\cdots,\degAInCFive$ in the previous equation and summing over $\degaInCFive$ yields
\begin{align}
J_{v_0-1}(t)J_{v_0}(t) - J_{v_0+\degAInCFive}(t)J_{v_0+\degAInCFive+1}(t)= \frac{\diffsymbol}{\diffsymbol t}\{\sum_{\degaInCFive=0}^{\degAInCFive} J^2_{v_0+\degaInCFive}(t)\}.
\end{align}
Applying~\eqref{eqnC5:recurrenceformula-1} with $v=v_0$ and~\eqref{eqnC5:recurrenceformula-2} with $v=v_0+\degAInCFive$, we get
\begin{align}
\frac{1}{2} \frac{\diffsymbol}{\diffsymbol t}\{ J^2_{v_0}(t)+J^2_{v_0+\degAInCFive}(t)\}+ \frac{v_0}{t}J_{v_0}^2(t) -\frac{v_0+\degAInCFive}{t}J_{v_0+\degAInCFive}^2(t)= \frac{\diffsymbol}{\diffsymbol t}\{\sum_{\degaInCFive=0}^{\degAInCFive} J^2_{v_0+\degaInCFive}(t)\}.
\end{align}
Integrating both sides from $x$ to $\infty$ gives
\begin{align}\label{eqnC5:4termsofKeyeuquation}
\frac{1}{2}  J^2_{v_0}(x)+ \frac{1}{2} J^2_{v_0+\degAInCFive}(x)- v_0 \int_{x}^{\infty}\frac{J_{v_0}^2(t)}{t}\diffsymbol t +(v_0+\degAInCFive)\int_{x}^{\infty}\frac{J_{v_0+\degAInCFive}^2(t)}{t}\diffsymbol t= \sum_{\degaInCFive=0}^{\degAInCFive} J^2_{v_0+\degaInCFive}(x).
\end{align}

Now set $x=\degAInCFive r$ and let $\degAInCFive\to\infty$ (so that $x\to\infty$). Since $v_0$ is fixed, the first and third terms converge to zero by~\eqref{eqnC5:approxJv1}, and the second term converges to zero by~\eqref{bound:maxJv}. Thus, by~\eqref{bound:maxJv2},
\begin{align}\label{eqnC5:equvalientlimt}
\lim_{ \degAInCFive\to\infty} \sum_{\degaInCFive=0}^{\degAInCFive} J_{v_0+\degaInCFive}^2(\degAInCFive r) &=\lim_{v\to\infty} v\int_{vr}^{\infty}\frac{J_{v}^2(t)}{t}\diffsymbol t= \lim_{v\to\infty} \int_{r}^{\infty}\frac{vJ_{v}^2(vt)}{t} \diffsymbol t\\ \nonumber
&=\frac{1}{2}-\lim_{v\to\infty} v\int_{0}^{vr}\frac{J_{v}^2(t)}{t} \diffsymbol t=\frac{1}{2}-\lim_{v\to\infty} \int_{0}^{r}\frac{vJ_{v}^2(vt)}{t} \diffsymbol t,
\end{align} 
if the limits exist.

We now show that~\eqref{eqnC5:equvalientlimt} converges to $\functionInCFiveFA(r)$. Since $\functionInCFiveFA(r)$ behaves differently for $r\leq 1$ and $r>1$, we discuss the two ranges separately.

Suppose that $r\leq1$. For any $r_0<2/e$, we split the integral on the right-hand side of~\eqref{eqnC5:equvalientlimt} as
\begin{align}\label{eqnC5:equvalientlimt_rlt1_1}
\int_{0}^{r}\frac{vJ_{v}^2(vt)}{t} \diffsymbol t \leq \int_{0}^{r_0}\frac{vJ_{v}^2(vt)}{t} \diffsymbol t +\int_{r_0}^{r}\frac{vJ_{v}^2(vt)}{t} \diffsymbol t .
\end{align}
On the first interval, the product representation~\eqref{eqnC5:multiplicationJv} gives
\begin{align}\label{eqnC5:equvalientlimt_rlt1_2}
\int_{0}^{r_0} \frac{vJ^2_v(vt)}{t} \diffsymbol t \leq\int_{0}^{r_0} \frac{v}{t}[\frac{(vt)^v}{2^v\substigamma(v+1)}]^2dt= \int_{0}^{r_0} \frac{v^{2v+1} t^{2v-1}}{2^{2v}\substigamma^2(v+1)}\diffsymbol t=\frac{1}{2} \left[\frac{((vr_0)/2)^v}{\substigamma(v+1)}\right]^2,
\end{align}
which converges to zero as $v\to\infty$ by Stirling's formula.

On the second interval $[r_0,r]$, the integral is controlled by $\int_{0}^{1}vtJ_v^2(vt)\diffsymbol t$, which has a closed-form expression (see, e.g., \cite[Eq. (6.52)]{Bowman}).
\begin{align}\label{eqnC5:equvalientlimt_rlt1_3}
\int_{r_0}^{r} \frac{vJ^2_v(vt)}{t} \diffsymbol t \leq (r_0)^{-2}v\int_{0}^{1} t J^2_v(vt)\diffsymbol t&=(r_0)^{-2}\frac{v}{2}[J^{'}_{v}(v)]^2=\frac{v[J^{'}_{v}(v)]^2 }{2r_0^2},
\end{align}
which is of order $\bigo(\tfrac{v}{(v^{2/3})^2})=\bigo(v^{-1/3})$ by~\eqref{bound:derivativeBesselj0to1}, and therefore converges uniformly to zero. Combining the two estimates gives
\begin{align*}
\frac{1}{2}-\lim_{v\to\infty} \int_{0}^{r}\frac{J_{v}^2(vt)}{t} \diffsymbol t =\frac{1}{2}-0 =\functionInCFiveFA(r),
\end{align*}
uniformly for $r\leq 1$.

For $r>1$, we use the substitution $t=\sec\beta$. By \cite[p.~382]{Olver:1997:ASF} or \cite[Chap. 8.41, Eq. (4)]{Watson},
\begin{align}\label{eqnC5:secsubstitute}
J_v(v\sec \beta)= \sqrt{\frac{2}{v \pi \tan \beta} }\left\{\cos[v(\tan\beta-\beta)-\frac{\pi}{4}]+\bigo(v^{-1})\right\},
\end{align}
which is valid uniformly for $t$ in any compact subset of $(1,\infty)$ (equivalently, for $\beta$ in a compact subset of $(0,\tfrac{\pi}{2})$). Thus, for fixed $r_0>1$,
\begin{align}
\lim_{v\to\infty}\int_{r}^{r_0}\frac{vJ_{v}^2(vt)}{t} \diffsymbol t& = \lim_{v\to\infty}\int_{\mathrm{arcsec} (r)}^{\mathrm{arcsec}(r_0)}\frac{2}{\pi} \cos^2\{v(\tan\beta-\beta-\frac{\pi}{4})\} \diffsymbol \beta +\bigo(v^{-1})\\ \nonumber
&=\frac{1}{\pi}\{\mathrm{arcsec}(r_0)-\mathrm{arcsec}(r)\}.
\end{align} 
Now, for $s>1$ and $r_0>s$, we have
$\int_{r_0}^{\infty}\frac{vJ_{v}^2(vt)}{t}dt \leq \int_{r_0}^{\infty} \frac{2 {C}_{s} v}{\pi vt^2}dt =\frac{2{C}_{s}}{\pi r_0}$ uniformly for $v\to\infty$, by~\eqref{bound:maxJv_Larget}. Thus
\begin{align*}
\lim_{v\to\infty}\int_{r}^{r_0}\frac{vJ_{v}^2(vt)}{t} \diffsymbol t = \frac{1}{\pi}\{\mathrm{arcsec}(r_0)-\mathrm{arcsec}(r)\} +\bigo(r_0^{-1}).
\end{align*} 
Letting $r_0\to \infty$ and using $\mathrm{arcsec}(r_0)\to \tfrac{\pi}{2}$ gives the desired estimate.
\end{proof}

\begin{rem}\label{remC5:remarkofKeyequation}
We also require uniform convergence in~\eqref{eqnC5:keyequation} on compact intervals $[a,1]$ and $[b,c]$ in $\R_{+}$, where $0<a<1<b<c<\infty$. These intervals are treated separately because, when $r=0$,
the first term on the left-hand side of~\eqref{eqnC5:4termsofKeyeuquation} becomes $J_0(0)=1$ for $v_0=0$, so~\eqref{eqnC5:keyequation} does not hold at $r=0$; see Figure \ref{fig:DemoForD2}. Near $r=1$, the term $vJ_v^2(vt)$ is not uniformly bounded near $t=1$; see~\eqref{bound:maxJv}. The first three terms on the left-hand side of~\eqref{eqnC5:4termsofKeyeuquation} converge uniformly on any compact subset of $(0,\infty)$. On the two intervals mentioned above, the uniform convergence in~\eqref{eqnC5:keyequation} therefore depends only on the convergence of $\lim_{v\to\infty}\int_{0}^{r}\frac{vJ_{v}^2(vt)}{t} \diffsymbol t$. For $[a,1]$, uniform convergence was already discussed in~\eqref{eqnC5:equvalientlimt_rlt1_1}-\eqref{eqnC5:equvalientlimt_rlt1_3}. On $[b,c]$, uniform convergence in~\eqref{eqnC5:keyequation} follows from a corollary of the Arzelà--Ascoli theorem, using the fact that $\int_{0}^{r}\frac{vJ_{v}^2(vt)}{t} \diffsymbol t$ is uniformly equicontinuous on $[b,c]$ because its integrand is uniformly bounded by~\eqref{bound:maxJv_Larget}. Further estimates for the convergence and its rate are given in \cite{Dunster2017}, using Airy functions.

We will also use the bound $\sum_{\degaInCFive=0}^{\degAInCFive}J_{v_0+\degaInCFive}^2(x)\leq1$, valid for all $v_0\geq0$, $\degAInCFive\in\N_0$, and $x\geq0$ (for example, by combining~\eqref{eqnC5:4termsofKeyeuquation}, with $\degAInCFive\to\infty$, and~\eqref{bound:maxJv2}).
\end{rem}

With~\eqref{keyequation} in hand, we can derive the asymptotic behavior of $\KLK(x,x)$. For any fixed $x$, $\KLK(x,x)$ grows as $\degAInCFive$ and $K$ increase; however, the growth rate is proportional to $K^d$. In particular, we obtain the following convergence result for the $K^{-d}$-normalized diagonal kernel.
\begin{thm}\label{thm:characterizeKLKdim2}
Let $\KLK$ be the reproducing kernel of $\spaceInCFive_{\degAInCFive, K}$ on $\R^2$. Let $\{\degAInCFive_i\}_{i\in\N}\subset\N_0$ and $\{K_i\}_{i\in\N}\subset\R_+$ satisfy $\degAInCFive_i\to\infty$, $K_i\to\infty$, and $\degAInCFive_i/K_i\to\kappa\in(0,\infty)$. Then
\begin{align}\label{eqnC5:limitdiagonalKd=2}
\lim_{i\to\infty}\frac{\mathcal K_{\degAInCFive_i,K_i}(x,x)}{K_i^2}
= \frac{1}{\pi} \frac{1}{\|x\|^2}\int_{0}^{\|x\|} t\functionInCFiveFA_{\langle \kappa \rangle}(t) \diffsymbol t.
\end{align}
At $x=0$, the right-hand side is understood by continuous extension, with value $1/(4\pi)$.
\end{thm}
\begin{proof}
We first assume $x\neq0$ and prove the assertion under the exact relation $\degAInCFive=\kappa K$, with $\kappa K\in\N_0$.
Taking $d=2$ in~\eqref{eqnC5:diagonalK}, we have
\begin{align} \label{eqnC5:analysisdiagonalKDim2}
\KLK(x,x)&=\frac{1}{2\pi}\int_{0}^{K}  \sum_{\degaInCFive =0}^{\degAInCFive} c_{\degaInCFive} J_{\degaInCFive}^2(k\|x\|) k \diffsymbol k\\ \nonumber
&=\frac{1}{\pi}\int_{0}^{K}  \sum_{\degaInCFive =0}^{\degAInCFive}  J_{\degaInCFive}^2(\frac{k}{\degAInCFive}\degAInCFive\|x\|) k \diffsymbol k -\frac{1}{2\pi}\int_{0}^{K}   J_0^2(k\|x\|) k \diffsymbol k,
\end{align}
where $c_{\degaInCFive}=\mathrm{dim}(H_{\degaInCFive}^2)$ equals $1$ for $\degaInCFive=0$ and $2$ otherwise.

Dividing both sides of~\eqref{eqnC5:analysisdiagonalKDim2} by $K^2$ and taking the limit $K\to\infty$, the second term on the right-hand side of~\eqref{eqnC5:analysisdiagonalKDim2} is of order $\bigo(K)$ before normalization and therefore vanishes after division by $K^2$. We obtain
\begin{align} \label{eqnC5:analysisdiagonalKDim2_2}
\lim_{\degAInCFive,K\to\infty; \frac{\degAInCFive}{K}= \kappa}\frac{\KLK(x,x)}{K^2}&=\lim_{\degAInCFive,K\to\infty; \frac{\degAInCFive}{K}= \kappa}\frac{1}{\pi K^2}\int_{0}^{K}  \sum_{\degaInCFive =0}^{\degAInCFive}  J_{\degaInCFive}^2(k\|x\|) k \diffsymbol k \\ 
\nonumber	&\overset{\scriptstyle t=(\frac{\|x\|}{K})k}{=} \lim_{K\to \infty}\frac{1}{\pi K^2}\int_{0}^{\|x\|} \sum_{\degaInCFive =0}^{\kappa K}  J_{\degaInCFive}^2(Kt) \left(\frac{K}{\|x\|} t\right)\left(\frac{K}{\|x\|}\right)\diffsymbol t \\
\nonumber &=\lim_{K\to \infty}\frac{1}{\pi K^2}\int_{0}^{\|x\|} \left[\sum_{\degaInCFive =0}^{\kappa K}  J_{\degaInCFive}^2(\kappa K\tfrac{t}{\kappa})\right] \left(\frac{K}{\|x\|} t\right) \left(\frac{K}{\|x\|}\right) \diffsymbol t \\
\nonumber &= \lim_{K\to \infty}\frac{1}{\pi \|x\|^2}\int_{0}^{\|x\|} \sum_{\degaInCFive =0}^{\kappa K}  J_{\degaInCFive}^2(\kappa K\tfrac{t}{\kappa})t\diffsymbol t  \\
\nonumber&= \frac{1}{\pi \|x\|^2}\int_{0}^{\|x\|} t\functionInCFiveFA(\tfrac{t}{\kappa}) \diffsymbol t=  \frac{1}{\pi \|x\|^2}\int_{0}^{\|x\|} t\functionInCFiveFA_{\langle \kappa \rangle}(t) \diffsymbol t.
\end{align}
In the penultimate equality, we use the Lebesgue dominated convergence theorem, together with~\eqref{eqnC5:keyequation} for $t>0$ and the uniform bound $\sum_{\degaInCFive=0}^{\degAInCFive} J_{v_0+\degaInCFive}^2(x)\leq 1$; the value at $t=0$ is irrelevant for the integral.

We now pass to the limit case $\degAInCFive_i/K_i\to\kappa\in(0,\infty)$. For $0<\eta<\kappa$, set
\[
\begin{aligned}
\degAInCFive_i^-&:=\lfloor(\kappa-\eta)K_i\rfloor,
&K_i^-&:=\degAInCFive_i^-/(\kappa-\eta),\\
\degAInCFive_i^+&:=\lceil(\kappa+\eta)K_i\rceil,
&K_i^+&:=\degAInCFive_i^+/(\kappa+\eta).
\end{aligned}
\]
For all sufficiently large $i$, these pairs bound
$(\degAInCFive_i,K_i)$ componentwise, and $K_i^\pm/K_i\to1$.
Denote the exact-ratio limit above by $G_x(\kappa)$. Monotonicity in both
cutoffs then gives
\[
G_x(\kappa-\eta)\leq
\liminf_{i\to\infty}\frac{\mathcal K_{\degAInCFive_i,K_i}(x,x)}{K_i^2}
\leq\limsup_{i\to\infty}\frac{\mathcal K_{\degAInCFive_i,K_i}(x,x)}{K_i^2}
\leq G_x(\kappa+\eta).
\]
Letting $\eta\downarrow0$ and using the continuity of $G_x$, which follows
from dominated convergence, proves~\eqref{eqnC5:limitdiagonalKd=2}.
At $x=0$, the result follows directly from~\eqref{eqnC5:diagonalK}, since only the degree-zero term contributes and $\mathcal K_{\degAInCFive,K}(0,0)=K^2/(4\pi)$.
\end{proof}
\begin{rem}
Observe that
\begin{align*}
\frac{1}{\pi \|x\|^2}\int_{0}^{\|x\|} t\functionInCFiveFA_{\langle \kappa \rangle}(t) \diffsymbol t = \frac{1}{\pi \|\kappa^{-1}x\|^2}\int_{0}^{\|\kappa^{-1}x\|} t\functionInCFiveFA(t) \diffsymbol t.
\end{align*}
Thus, for the exact-ratio subfamily,
\begin{align*}
\lim_{\degAInCFive,K\to\infty; \frac{\degAInCFive}{K}= \kappa}\frac{\KLK(x,x)}{K^2}=	\lim_{K\to\infty}\frac{\mathcal{K}_{K,K}(\kappa^{-1}x,\kappa^{-1}x)}{K^2}.
\end{align*}

\end{rem}
\begin{figure}[!htbp]
\centering
\begin{subfigure}[t]{0.48\textwidth}
\centering
\includegraphics[width=\linewidth]{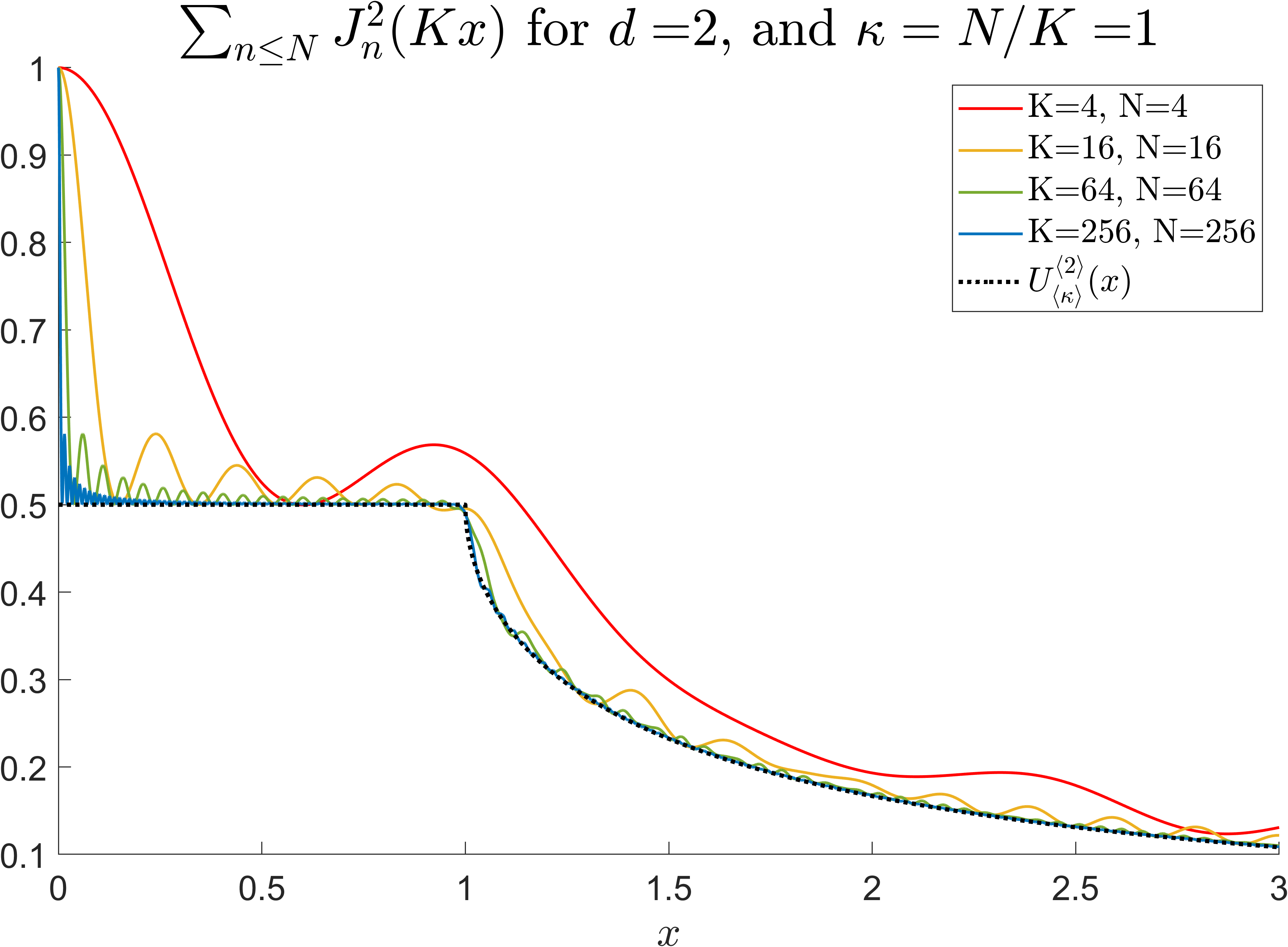}
\caption{Bessel sum, $\kappa=1$.}
\end{subfigure}
\hfill
\begin{subfigure}[t]{0.48\textwidth}
\centering
\includegraphics[width=\linewidth]{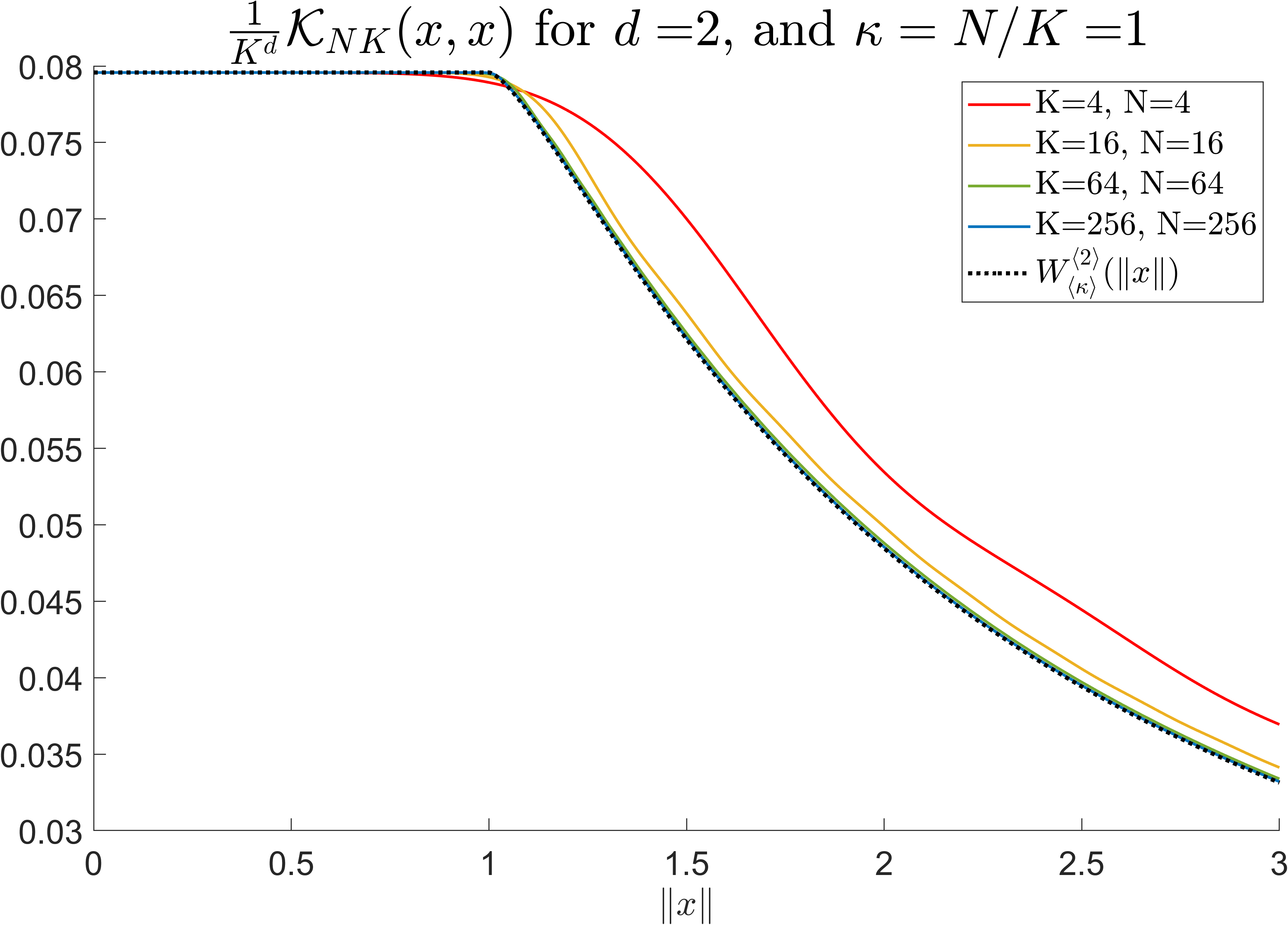}
\caption{Normalized diagonal kernel, $\kappa=1$.}
\end{subfigure}

\medskip
\begin{subfigure}[t]{0.48\textwidth}
\centering
\includegraphics[width=\linewidth]{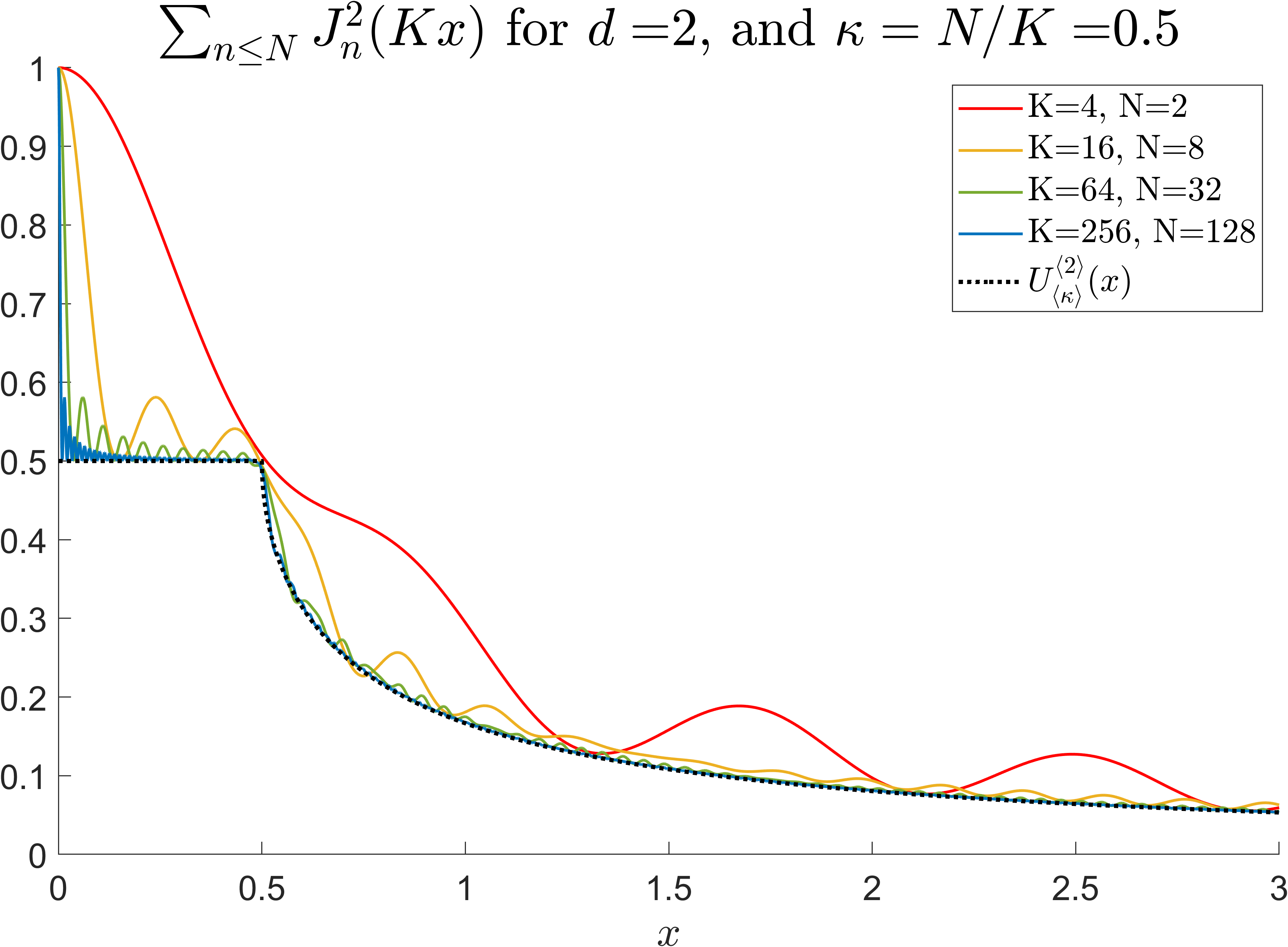}
\caption{Bessel sum, $\kappa=0.5$.}
\end{subfigure}
\hfill
\begin{subfigure}[t]{0.48\textwidth}
\centering
\includegraphics[width=\linewidth]{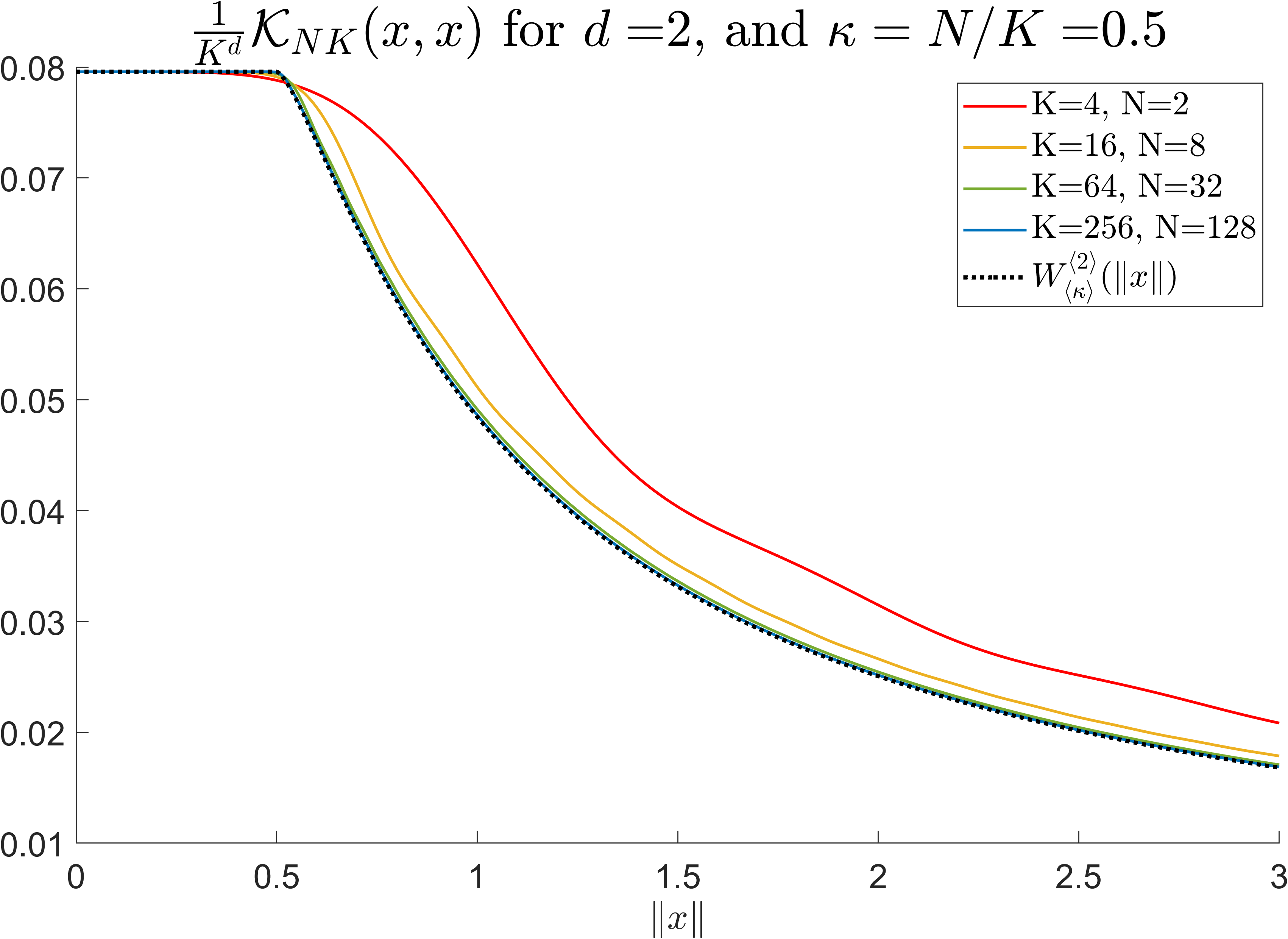}
\caption{Normalized diagonal kernel, $\kappa=0.5$.}
\end{subfigure}
\caption{Numerical convergence in dimension $d=2$. The left column illustrates the finite Bessel sums in~\eqref{eqnC5:keyequation}, and the right column illustrates the normalized diagonal kernel in~\eqref{eqnC5:limitdiagonalKd=2}. The top and bottom rows correspond to $\kappa=1$ and $\kappa=0.5$, respectively; the dotted curves are the limiting profiles.}
\label{fig:DemoForD2}
\end{figure}
\FloatBarrier

\subsection{The case $d\geq3$}

In this subsection, we extend Theorem \ref{thm:characterizeKLKdim2} to higher dimensions. The starting point is still the observation~\eqref{eqnC5:keyequation}, with some necessary modifications. In the rest of the paper, the role of the dimension $d$ will be emphasized. We define several functions and quantities related to $d$, using the superscript ${}^{\langle d \rangle}$ to indicate this dependence; this notation should not be confused with a power or a $d$-th derivative. 

We first recall the formula for $\KLK$ in dimension $d\geq3$. Writing ${C}^{\langle d\rangle}_{\degaInCFive}:=\mathrm{dim}(H_{\degaInCFive}^{d})$, we have
\begin{align}\label{eqnC5:defK\degAInCFive, Kxx}
\KLK(x,x)=\frac{(r_x)^{2-d}}{\mathrm{vol}(\Sphered)}\int_{0}^{K}  \sum_{\degaInCFive=0}^{\degAInCFive} {C}^{\langle d \rangle}_{\degaInCFive} J^2_{\degaInCFive+\frac{d-2}{2}}(kr_x) k \diffsymbol k.
\end{align}
Compared with the previous subsection, ${C}^{\langle d \rangle}_{\degaInCFive}$ is slightly more complicated for $d\geq3$. We need the following leading-term estimate:
\begin{align}\label{eqnC5:asympoticofdCl}
{C}^{\langle d \rangle}_{\degaInCFive}&=\binom{\degaInCFive+d-2}{d-2}+\binom{\degaInCFive+d-3}{d-2} =\frac{2}{\substigamma(d-1)}{\degaInCFive}^{d-2}+\bigo({\degaInCFive}^{d-3}).
\end{align}
Similarly, for the difference of ${C}^{\langle d \rangle}_{\degaInCFive}$, we have
\begin{align}
\nonumber	{C}^{\langle d \rangle}_{\degaInCFive+1}-{C}^{\langle d \rangle}_{\degaInCFive}
&=\binom{\degaInCFive+d-1}{d-2}-\binom{\degaInCFive+d-3}{d-2}=\binom{\degaInCFive+d-2}{d-3}+\binom{\degaInCFive+d-3}{d-3}\\
& =\frac{2}{\substigamma(d-2)}{\degaInCFive}^{d-3}+\bigo({\degaInCFive}^{d-4}).
\end{align}
Applying Abel's transformation (summation by parts) to the integrand in~\eqref{eqnC5:defK\degAInCFive, Kxx} yields
\begin{align}\label{eqnC5:analysiDimd_1}
\nonumber \sum_{\degaInCFive=0}^{\degAInCFive} {C}^{\langle d \rangle}_{\degaInCFive} J^2_{\degaInCFive+\frac{d-2}{2}}(s)& = \sum_{\degaInCFive=0}^{\degAInCFive} \left[({C}^{\langle d \rangle}_{\degaInCFive}-{C}^{\langle d \rangle}_{\degaInCFive-1}) \sum_{{\degaInCFive}^{'}=\degaInCFive}^{\degAInCFive} J^2_{{\degaInCFive}^{'}+\frac{d-2}{2}}(s)\right] \\
	&={C}^{\langle d \rangle}_{\degAInCFive}\sum_{\degaInCFive=0}^{\degAInCFive}J^2_{\degaInCFive+\frac{d-2}{2}}(s) - \sum_{\degaInCFive=0}^{\degAInCFive} \left[({C}^{\langle d \rangle}_{\degaInCFive}-{C}^{\langle d \rangle}_{\degaInCFive-1}) \sum_{{\degaInCFive}^{'}=0}^{\degaInCFive-1} J^2_{{\degaInCFive}^{'}+\frac{d-2}{2}}(s)\right],
\end{align}
We next motivate the expected limit of~\eqref{eqnC5:analysiDimd_1}. Set $s=\degAInCFive r$ for fixed $r>0$ and let $\degAInCFive\to\infty$. By~\eqref{eqnC5:keyequation}, the first sum on the right-hand side converges to $\functionInCFiveFA(r)$. Similarly, the inner sum in the second term may be heuristically replaced by $\functionInCFiveFA(\frac{\degAInCFive}{\degaInCFive-1}r)$. After division by ${C}^{\langle d \rangle}_{\degAInCFive}$, this suggests the following limit; a complete proof is given in Appendix~\ref{secC5:appendix}.
\begin{align}\label{eqnC5:analysiDimd_2}
\lim_{\degAInCFive\to\infty}\frac{\sum_{\degaInCFive=0}^{\degAInCFive} {C}^{\langle d \rangle}_{\degaInCFive} J^2_{\degaInCFive+\frac{d-2}{2}}(\degAInCFive r)}{{C}^{\langle d \rangle}_{\degAInCFive}}& = \functionInCFiveFA(r)-\lim_{\degAInCFive\to\infty}\sum_{\degaInCFive =0}^{\degAInCFive-1} \frac{{C}^{\langle d \rangle}_{\degaInCFive+1}-{C}^{\langle d \rangle}_{\degaInCFive}}{{C}^{\langle d \rangle}_{\degAInCFive}}\functionInCFiveFA\left(\tfrac{\degAInCFive}{\degaInCFive}r\right).
\end{align}
To simplify this expression further, observe that the sum on the right-hand side of~\eqref{eqnC5:analysiDimd_2} can be interpreted as a discrete approximation of $\int_{1}^{\infty} \functionInCFiveFA(tr)p(t) \diffsymbol t$. To identify the weight function $p(t)$, we replace ${C}^{\langle d \rangle}_{\degaInCFive+1}-{C}^{\langle d \rangle}_{\degaInCFive}$ and ${C}^{\langle d \rangle}_{\degAInCFive}$ by their leading terms $\frac{2}{\substigamma(d-2)}{\degaInCFive}^{d-3}$ and $\frac{2}{\substigamma(d-1)}{\degAInCFive}^{d-2}$, respectively; this replacement does not change the limit.
The desired weight function $p(t)$ should therefore satisfy
\begin{align}
\int_{1}^{T}p(t)\diffsymbol t \propto \lim_{\degAInCFive\to\infty}\sum_{\degaInCFive=\lceil \frac{\degAInCFive}{T}\rceil}^{\degAInCFive} \frac{{\degaInCFive}^{d-3}}{{\degAInCFive}^{d-2}}= \int_{\frac{1}{T}}^{1}t^{d-3}\diffsymbol t  \quad\text{and}\quad \int_{1}^{\infty} p(t)\diffsymbol t=1,
\end{align}
where $\propto$ denotes proportionality as functions of $T$. This gives
\begin{align}
p(t)= (d-2)t^{1-d}.
\end{align}
The preceding argument predicts the limit
\begin{align}\label{eqnC5:summation2integral}
\lim_{\degAInCFive\to\infty}\sum_{\degaInCFive =0}^{\degAInCFive-1} \frac{{C}^{\langle d \rangle}_{\degaInCFive+1}-{C}^{\langle d \rangle}_{\degaInCFive}}{{C}^{\langle d \rangle}_{\degAInCFive}}\functionInCFiveFA\left(\tfrac{\degAInCFive}{\degaInCFive}r\right) = \int_{1}^{\infty}\functionInCFiveFA(tr) \left[(d-2)t^{1-d}\right] \diffsymbol t.
\end{align}

The rigorous statement, proved in Appendix~\ref{secC5:appendix}, is the following higher-dimensional version of~\eqref{eqnC5:keyequation}.

\begin{prop} Let the dimension satisfy $d\geq 3$, and let $U(\cdot)$ be the function defined in~\eqref{eqn:defU}. Then, for any $r>0$,
\begin{align}\label{eqnC5:summation2integral2}
\lim_{\degAInCFive\to\infty}\frac{\sum_{\degaInCFive=0}^{\degAInCFive} \mathrm{dim}(H_{\degaInCFive}^{d}) J^2_{\degaInCFive+\frac{d-2}{2}}(\degAInCFive r)}{\mathrm{dim}(H_{\degAInCFive}^{d})} = \functionInCFiveFA(r)-\int_{1}^{\infty}\functionInCFiveFA(tr) \left[(d-2)t^{1-d}\right] \diffsymbol t.
\end{align}
\end{prop}
We now compute the integral $\int_{1}^{\infty}\functionInCFiveFA(tr)p(t)\diffsymbol t$. Since $\functionInCFiveFA(\cdot)$ is not differentiable at $1$, it is convenient to discuss the cases $r\geq 1$ and $0<r<1$ separately.
\paragraph{} When $r\geq 1$, we have $tr\geq 1$ for $t\geq 1$. In this region, $\functionInCFiveFA(tr)$ is differentiable as a function of $t$, with
\begin{align*}
\frac{\diffsymbol}{\diffsymbol t}\functionInCFiveFA(tr)=\frac{\diffsymbol}{\diffsymbol t}\left\{-\frac{1}{\pi}\mathrm{arcsec}(tr)\right\}=\frac{(-1)r}{\pi tr\sqrt{t^2r^2-1}}.
\end{align*}
Integration by parts gives
\begin{align}\label{eqn:calculatingU_1}
&\quad \int_{1}^{\infty} \functionInCFiveFA(tr)(d-2)t^{1-d}\diffsymbol t\\
\nonumber	&=\left.[-\functionInCFiveFA(tr)t^{2-d}]\right|_{t=1}^{\infty} -\int_{1}^{\infty} \frac{(-1)r}{\pi tr\sqrt{t^2r^2-1}} (-1) t^{2-d}\diffsymbol t \\
\nonumber	&=\functionInCFiveFA(r)- \frac{1}{\pi}\int_{1}^{\infty}\frac{t^{1-d}}{\sqrt{t^2r^2-1}}\diffsymbol t.
\end{align}

\paragraph{}When $r<1$, we split the domain of integration into the two intervals $[1,1/r]$ and $(1/r,\infty)$, obtaining
\begin{align}\label{eqn:calculatingU_2}
\nonumber	\int_{1}^{\infty} \functionInCFiveFA(tr)(d-2)t^{1-d}\diffsymbol t & =\int_{1}^{\frac{1}{r}}\functionInCFiveFA(tr)(d-2)t^{1-d}\diffsymbol t +\int_{\frac{1}{r}}^{\infty}\functionInCFiveFA(tr)(d-2)t^{1-d}\diffsymbol t  \\ 
\nonumber	&= \frac{1}{2}(1- r^{d-2})+ r^{d-2}\int_{1}^{\infty}\functionInCFiveFA(z)(d-2)z^{1-d} \diffsymbol z\\
\nonumber	&=\frac{1}{2} -\frac{r^{d-2}}{\pi} \int_{1}^{\infty}\frac{t^{1-d}}{\sqrt{t^2-1}}\diffsymbol t\\
&=\functionInCFiveFA(r) -\frac{r^{d-2}}{\pi} \int_{1}^{\infty}\frac{t^{1-d}}{\sqrt{t^2-1}}\diffsymbol t,
\end{align}
where the third equality uses~\eqref{eqn:calculatingU_1}, and the last equality uses $\functionInCFiveFA(r)=1/2$ for $r<1$.

\hfill
Motivated by the preceding analysis, we define ${\functionInCFiveFA}^{\langle d\rangle}$ and ${\functionInCFiveFB}^{\langle d\rangle}$ for integers $d\geq2$.
\begin{defi}\label{def:UandW} For integer $d\geq 2$, we define ${\functionInCFiveFA}^{\langle d \rangle}: [0,\infty)\to\R$ as
\begin{subnumcases}{
	{\functionInCFiveFA}^{\langle d \rangle}(r) := \functionInCFiveFA(r)-\int_{1}^{\infty} \functionInCFiveFA(tr)(d-2)t^{1-d}\diffsymbol t\\ 
=} 
\frac{r^{d-2}}{\pi} \int_{1}^{\infty}\frac{t^{1-d}}{\sqrt{t^2-1}}\diffsymbol t \quad ( r\leq 1), \label{eqnC5:defAuxFun1}\\ 
\frac{1}{\pi}\int_{1}^{\infty}\frac{t^{1-d}}{\sqrt{r^2t^2-1}}\diffsymbol t \quad (r>1). \label{eqnC5:defdauxFun2} 
\end{subnumcases}
Moreover, for $r>0$, we define ${\functionInCFiveFB}^{\langle d \rangle}$ by
\begin{align}\label{eqnC5:defdWforSHB}
{\functionInCFiveFB}^{\langle d \rangle}(r):=\frac{2}{\substigamma(d-1)\mathrm{vol}(\Sphered)} \frac{1}{r^{d}}\int_{0}^{r} {\functionInCFiveFA}^{\langle d \rangle}(t) t\diffsymbol t.
\end{align}
At $r=0$, we define it by continuous extension:
${\functionInCFiveFB}^{\langle d\rangle}(0)
:=\lim_{r\to 0}{\functionInCFiveFB}^{\langle d\rangle}(r)
=\frac{\mathrm{vol}(\BB^d)}{(2\pi)^d}.
$

\end{defi}	
%

\begin{rem}\label{rem:explicitformford23} When $d=2$, setting $d-2=0$ in the second term of~\eqref{eqnC5:defAuxFun1} gives
\begin{align*}
{\functionInCFiveFA}^{\langle 2 \rangle}(r)=\functionInCFiveFA(r),
\end{align*}
and ${\functionInCFiveFB}^{\langle 2 \rangle}(r)$ has the form
\begin{align}\label{eqnC5:def2W}
{\functionInCFiveFB}^{\langle 2 \rangle}(r)=
\begin{cases}
\frac{1}{4\pi} &\quad (r\leq 1),\\
\frac{1}{4\pi}+\frac{\sqrt{r^2-1}}{2\pi^2 r^2}-\frac{\arccos (1/r)}{2\pi^2} &\quad (r>1).
\end{cases} 
\end{align}
For $d=3$,
\begin{align}
\frac{1}{\pi}\int_{t=1}^{\infty}\frac{t^{1-3}}{\sqrt{t^2r^2-1}}\diffsymbol t= \left.\left[\frac{1}{\pi} \frac{\sqrt{r^2t^2-1}}{t}\right]\right|_{t=1}^{\infty}
=\frac{1}{\pi}(r-\sqrt{r^2-1}),
\end{align}
which yields
\begin{align}\label{eqnC5:def3\functionInCFiveFA}
{\functionInCFiveFA}^{\langle 3 \rangle}(r) = 
\begin{cases}
\frac{1}{\pi} r &\quad (r\leq 1),\\
\frac{1}{\pi}(r-\sqrt{r^2-1}) &\quad (r>1).
\end{cases}
\end{align}
and
\begin{align}\label{eqnC5:def3W}
{\functionInCFiveFB}^{\langle 3 \rangle}(r)=
\begin{cases}
\frac{1}{6\pi^2} &\quad (r\leq 1),\\
\frac{1}{6\pi^2}(1-\frac{(r^2-1)^{3/2}}{r^3}) &\quad (r>1).
\end{cases} 
\end{align}
\end{rem}
\begin{thm}\label{thm:convergenceofshbKtoW}
For fixed integer $d\geq 3$ and $x\in\R^d$ such that $\|x\|>0$, let $\{\degAInCFive_i\}_{i\in\N}\subset\N_0$ and $\{K_i\}_{i\in\N}\subset\R_+$ satisfy $\degAInCFive_i\to\infty$, $K_i\to\infty$, and $\degAInCFive_i/K_i\to\kappa\in(0,\infty)$. Then
\begin{align}\label{eqnC5:convergenceofshbKtoW}
\lim_{i\to\infty}\frac{\mathcal K_{\degAInCFive_i,K_i}(x,x)}{K_i^d}
= {\functionInCFiveFB}^{\langle d \rangle}_{\langle \kappa \rangle}(\|x\|).
\end{align}
\end{thm}
\begin{proof}
First assume the exact relation $\degAInCFive=\kappa K$ with $\kappa K\in\N_0$. Then
\begin{align} \label{eqnC5:analysisdiagonalKDimd}
\frac{\KLK(x,x)}{K^d} &=\frac{\|x\|^{2-d}}{\mathrm{vol}(\Sphered)K^d}\int_{0}^{K}  \sum_{\degaInCFive =0}^{\degAInCFive}  {C}^{\langle d \rangle}_{\degaInCFive} J^2_{\degaInCFive+\frac{d-2}{2}}(k\|x\|) k \diffsymbol k \\ \nonumber
&=\frac{\|x\|^{2-d} {C}^{\langle d \rangle}_{\kappa K}}{\mathrm{vol}(\Sphered)K^d}\int_{0}^{K}  \sum_{\degaInCFive =0}^{\kappa K}  \frac{{C}^{\langle d \rangle}_{\degaInCFive}}{{C}^{\langle d \rangle}_{\kappa K}} J^2_{\degaInCFive+\frac{d-2}{2}}(k\|x\|) k \diffsymbol k \\ \nonumber
&\overset{\scriptstyle t=(\frac{\|x\|}{K})k}{=} \frac{{C}^{\langle d \rangle}_{\kappa K}\|x\|^{2-d}}{\mathrm{vol}(\Sphered)K^d}\int_{0}^{\|x\|}  \left[\sum_{\degaInCFive =0}^{\kappa K}\frac{{C}^{\langle d \rangle}_{\degaInCFive}}{{C}^{\langle d \rangle}_{\kappa K}} J^2_{\degaInCFive+\frac{d-2}{2}}(\kappa K \tfrac{t}{\kappa})\right] \left(\frac{K}{\|x\|} t\right) \left(\frac{K}{\|x\|}\right)   \diffsymbol t \\ \nonumber
&=\frac{2  \kappa^{d-2}+\bigo(K^{-1}) }{\substigamma(d-1)\mathrm{vol}(\Sphered)}\frac{1}{\|x\|^d}\int_{0}^{\|x\|}  \left[\sum_{\degaInCFive =0}^{\kappa K}\frac{{C}^{\langle d \rangle}_{\degaInCFive}}{{C}^{\langle d \rangle}_{\kappa K}} J^2_{\degaInCFive+\frac{d-2}{2}}(\kappa K \tfrac{t}{\kappa})\right] t  \diffsymbol t.
\end{align}
Letting $K\to\infty$ and applying~\eqref{eqnC5:summation2integral2} to the bracketed expression on the right-hand side of~\eqref{eqnC5:analysisdiagonalKDimd}, we may interchange the limit and the integral by the dominated convergence theorem, since the bracketed expression is bounded by $1$. Thus
\begin{align}
\lim_{\degAInCFive,K\to\infty; \frac{\degAInCFive}{K}= \kappa}\frac{\KLK(x,x)}{K^d}	&= \frac{2 \kappa^{d-2}}{\substigamma(d-1)\mathrm{vol}(\Sphered)} \frac{1}{\|x\|^{d}}\int_{0}^{\|x\|} {\functionInCFiveFA}^{\langle d \rangle}(\tfrac{t}{\kappa}) t  \diffsymbol t\\ \nonumber
&= \frac{2}{\substigamma(d-1)\mathrm{vol}(\Sphered)} \frac{1}{\|\kappa^{-1}x\|^{d}}\int_{0}^{\|\kappa^{-1}x\|} {\functionInCFiveFA}^{\langle d \rangle}(t) t\diffsymbol t \\ \nonumber
&= {\functionInCFiveFB}^{\langle d \rangle}(\|\kappa^{-1}x\|) = {\functionInCFiveFB}^{\langle d \rangle}_{\langle \kappa \rangle}(\|x\|).
\end{align}

The general case $\degAInCFive_i/K_i\to\kappa$ follows from
the same monotonicity and squeezing argument as in the proof of
Theorem~\ref{thm:characterizeKLKdim2}. Indeed, applying the exact-ratio
result with ratios $\kappa\pm\eta$ and then letting $\eta\downarrow0$
yields the assertion.
\end{proof}

\begin{rem}\label{rem:extension_analysisdiagonalKDimd}
The endpoint regimes can be treated by the same monotonicity argument. If $\degAInCFive_i/K_i\to0$, the limiting profile is interpreted as ${\functionInCFiveFB}^{\langle d\rangle}_{\langle0\rangle}(r)\equiv0$ for $r>0$. If $\degAInCFive_i/K_i\to\infty$, it is interpreted as ${\functionInCFiveFB}^{\langle d\rangle}_{\langle\infty\rangle}(r)\equiv{\functionInCFiveFB}^{\langle d\rangle}(0)$.
\end{rem}

\subsection{\texorpdfstring{Qualitative structure of the limiting profile: Fourier and Hankel regimes}{Qualitative structure of the limiting profile: Fourier and Hankel regimes}}
\label{subsec:qualitative-density}

We now collect the qualitative consequences of the limiting profile obtained above. 
The explicit formulas for ${\functionInCFiveFA}^{\langle d \rangle}$ and ${\functionInCFiveFB}^{\langle d \rangle}$ in the cases $d=2,3$ were given in Remark~\ref{rem:explicitformford23}, and the numerical demonstrations of Theorems~\ref{thm:convergenceofshbKtoW} and~\ref{thm:characterizeKLKdim2} are shown in Figures~\ref{fig:DemoForD2} and~\ref{fig:DemoForD3}. 
For arbitrary dimension, however, the formulas~\eqref{eqnC5:defAuxFun1},~\eqref{eqnC5:defdauxFun2}, and~\eqref{eqnC5:defdWforSHB} already reveal the main structure of the limiting density. 
The next two corollaries identify its two structural features that drive the Shannon-number asymptotics: the Paley--Wiener plateau and the Hankel-type far-field tail. 

\begin{cor}\label{cor:profile-plateau-monotonicity}
Let $d\geq2$. Then $r^{2-d}{\functionInCFiveFA}^{\langle d\rangle}(r)$ and ${\functionInCFiveFB}^{\langle d\rangle}(r)$ are constant on $[0,1]$ and monotonically decreasing on $[1,\infty)$.
\end{cor}
\begin{proof} We first prove the assertion for $r^{2-d}{\functionInCFiveFA}^{\langle d \rangle}(r)$. From the formula for ${\functionInCFiveFA}^{\langle d \rangle}(r)$, the function $r^{2-d}{\functionInCFiveFA}^{\langle d \rangle}(r)$ is constant on $[0,1]$, so it remains to discuss its behavior on $(1,\infty)$. For $r>1$, we have
\begin{align}
r^{2-d}{\functionInCFiveFA}^{\langle d \rangle}(r)=\frac{1}{\pi}\int_{1}^{\infty}\frac{t^{1-d}r^{1-d}}{\sqrt{r^2t^2-1}}r\diffsymbol t
\overset{z=rt}{=}\frac{1}{\pi}\int_{r}^{\infty} \frac{z^{1-d}}{\sqrt{z^2-1}}\diffsymbol z,
\end{align}
Since the integrand is strictly positive, $r^{2-d}{\functionInCFiveFA}^{\langle d\rangle}(r)$ is strictly decreasing.

We now consider ${\functionInCFiveFB}^{\langle d \rangle}(r)$. For $0\leq r\leq 1$, the definition gives
\begin{align*}
{\functionInCFiveFB}^{\langle d \rangle}(r) \propto \frac{\int_{0}^{r}t^{d-1}\diffsymbol t}{r^d}=\frac{1}{d}
\end{align*}
which is constant in this region. For $r>1$, writing ${\functionInCFiveFA}^{\langle d\rangle}(t)t$ as $[t^{2-d}{\functionInCFiveFA}^{\langle d\rangle}(t)]t^{d-1}$ gives
\begin{align*}
{\functionInCFiveFB}^{\langle d \rangle}(r)\propto \frac{1}{r^d}\int_{0}^{r}[t^{2-d}{\functionInCFiveFA}^{\langle d \rangle}(t)]t^{d-1}\diffsymbol t,
\end{align*} 
which leads to
\begin{align}
\nonumber		\frac{\diffsymbol}{\diffsymbol r} \left\{\frac{1}{r^d}\int_{0}^{r}[t^{2-d}{\functionInCFiveFA}^{\langle d \rangle}(t)]t^{d-1}\diffsymbol t\right\}&=\frac{1}{r}\left\{ [r^{2-d}{\functionInCFiveFA}^{\langle d \rangle}(r)]-\frac{d}{r^d}\int_{0}^{r} t^{2-d}{\functionInCFiveFA}^{\langle d \rangle}(t) t^{d-1}\diffsymbol t\right\} \\
\nonumber		&= \frac{d}{r^{d+1}}\int_{0}^{r} \left[r^{2-d}{\functionInCFiveFA}^{\langle d \rangle}(r)-t^{2-d}{\functionInCFiveFA}^{\langle d \rangle}(t)\right] t^{d-1}\diffsymbol t\\
&<0,
\end{align}
where the bracketed expression is strictly negative by the monotonicity of $r^{2-d}{\functionInCFiveFA}^{\langle d\rangle}(r)$.
\end{proof}

The following corollary characterizes the far-field behavior of ${\functionInCFiveFB}^{\langle d\rangle}$.
\begin{cor}\label{cor:tailWHankel}
Let $d\geq2$. Then
\begin{align}
\lim_{r\to\infty} {\functionInCFiveFA}^{\langle d \rangle}(r) r=\frac{1}{(d-1)\pi}.
\end{align}
Consequently,
\begin{align}
\lim_{r\to\infty}\frac{	{\functionInCFiveFB}^{\langle d \rangle}(r) }{r^{1-d}} =\frac{2}{\substigamma(d) \pi\mathrm{vol}(\Sphered)}.
\end{align}
\end{cor}
\begin{proof} The case $d=2$ follows directly from the explicit formula. For $d\geq3$,
\begin{align}
{\functionInCFiveFA}^{\langle d \rangle}(r) r&=\frac{r}{\pi}\int_{1}^{\infty}\frac{t^{1-d}}{\sqrt{r^2t^2-1}}\diffsymbol t= \frac{1}{\pi}\int_{1}^{\infty}\frac{t^{1-d}}{\sqrt{t^2-1/r^2}}\diffsymbol t\\
\nonumber		&=\frac{1}{\pi}\int_{1}^{\infty} t^{-d}+ \frac{r^{-2}t^{1-d}}{t\sqrt{t^2-r^{-2}}(t+\sqrt{t^2-r^{-2}})} \diffsymbol t\\ \nonumber
\nonumber		&=\frac{1}{(d-1)\pi}+\frac{r^{-2}}{2(d+1)\pi}+\bigo(r^{-4}).
\end{align}
Substituting this expansion into~\eqref{eqnC5:defdWforSHB} gives
\begin{align}
\frac{	{\functionInCFiveFB}^{\langle d \rangle}(r) }{r^{1-d}}& = \frac{2}{\substigamma(d-1)\mathrm{vol}(\Sphered)} \frac{\int_{0}^{r}{\functionInCFiveFA}^{\langle d \rangle}(t)t\diffsymbol t}{r} \\
\nonumber	&=\frac{2}{\substigamma(d) \pi\mathrm{vol}(\Sphered)} + \frac{1}{\substigamma(d-1)(d+1)\pi\mathrm{vol}(\Sphered)}r^{-2} +\bigo(r^{-4}) \quad (r\to\infty).
\end{align}
\end{proof}

\begin{rem}[Fourier and Hankel regimes]\label{rem:fourier-hankel-regimes}
The preceding corollary gives a useful interpretation of the profile ${\functionInCFiveFB}^{\langle d\rangle}$ and of the parameter $\kappa$. The density appearing in Theorem~\ref{thm:convergenceofshbKtoW} is
\[
{\functionInCFiveFB}^{\langle d\rangle}_{\langle\kappa\rangle}(\|x\|)
={\functionInCFiveFB}^{\langle d\rangle}(\|x\|/\kappa).
\]
Since ${\functionInCFiveFB}^{\langle d\rangle}$ is constant on $[0,1]$, the SFB truncation has the same constant local density as in the classical Paley--Wiener setting on the region $\|x\|\leq \kappa$. In this regime the angular bandwidth is sufficiently large for the normalized diagonal kernel to recover the constant density of the Paley--Wiener projection. On the other hand, when $\|x\|\gg\kappa$, Corollary~\ref{cor:tailWHankel} gives
\[
{\functionInCFiveFB}^{\langle d\rangle}_{\langle\kappa\rangle}(\|x\|)
\sim
\frac{2}{\substigamma(d)\pi\mathrm{vol}(\Sphered)}\,
\kappa^{d-1}\|x\|^{1-d}.
\]
After multiplication by the spherical volume element $\|x\|^{d-1}\diffsymbol \|x\|\diffsymbol\omega$, this corresponds to a radial density of order $\kappa^{d-1}\diffsymbol\|x\|$. Thus, at the level of the leading local density, the far-field regime has a radial-length scaling modulated by the angular-bandwidth factor $\kappa^{d-1}$. This comparison is made at the level of the density law and is motivated by the Hankel concentration setting of Abreu and Bandeira~\cite{Abreu2012}: for Hankel band-limited spaces, the trace of the concentration operator on a radial interval has leading term proportional to the interval length times the Hankel spectral measure. In our case, the far-field formula shows the same type of radial-length scaling, with $\kappa^{d-1}$ playing the role of an angular-bandwidth prefactor. 
The function ${\functionInCFiveFB}^{\langle d\rangle}$ therefore provides a quantitative description of the transition from the Paley--Wiener local-density regime to a Hankel-type radial-density regime.
\end{rem}

\begin{figure}[!htbp]
\centering
\begin{subfigure}[t]{0.48\textwidth}
\centering
\includegraphics[width=\linewidth]{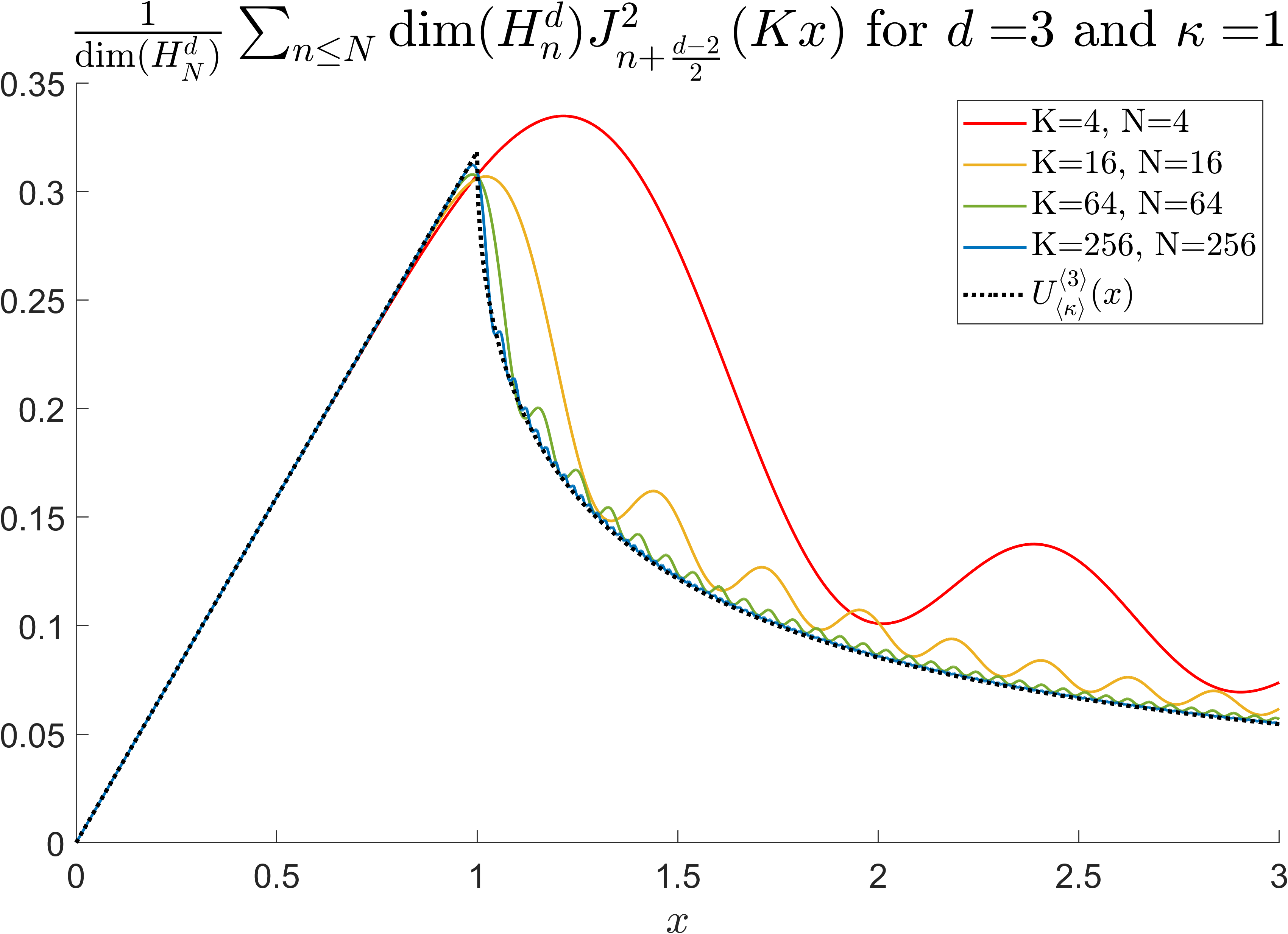}
\caption{Bessel sum, $\kappa=1$.}
\end{subfigure}
\hfill
\begin{subfigure}[t]{0.48\textwidth}
\centering
\includegraphics[width=\linewidth]{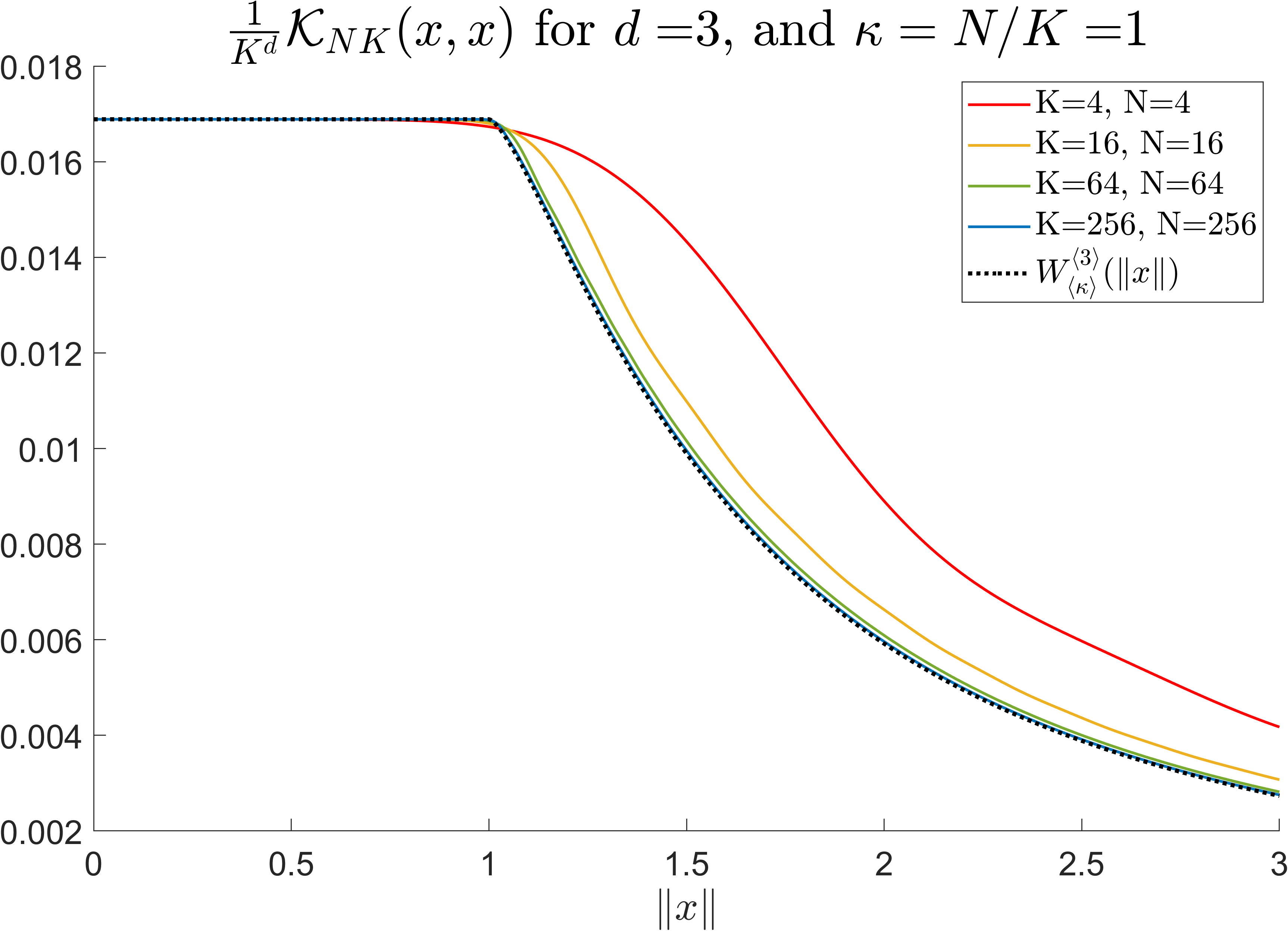}
\caption{Normalized diagonal kernel, $\kappa=1$.}
\end{subfigure}

\medskip
\begin{subfigure}[t]{0.48\textwidth}
\centering
\includegraphics[width=\linewidth]{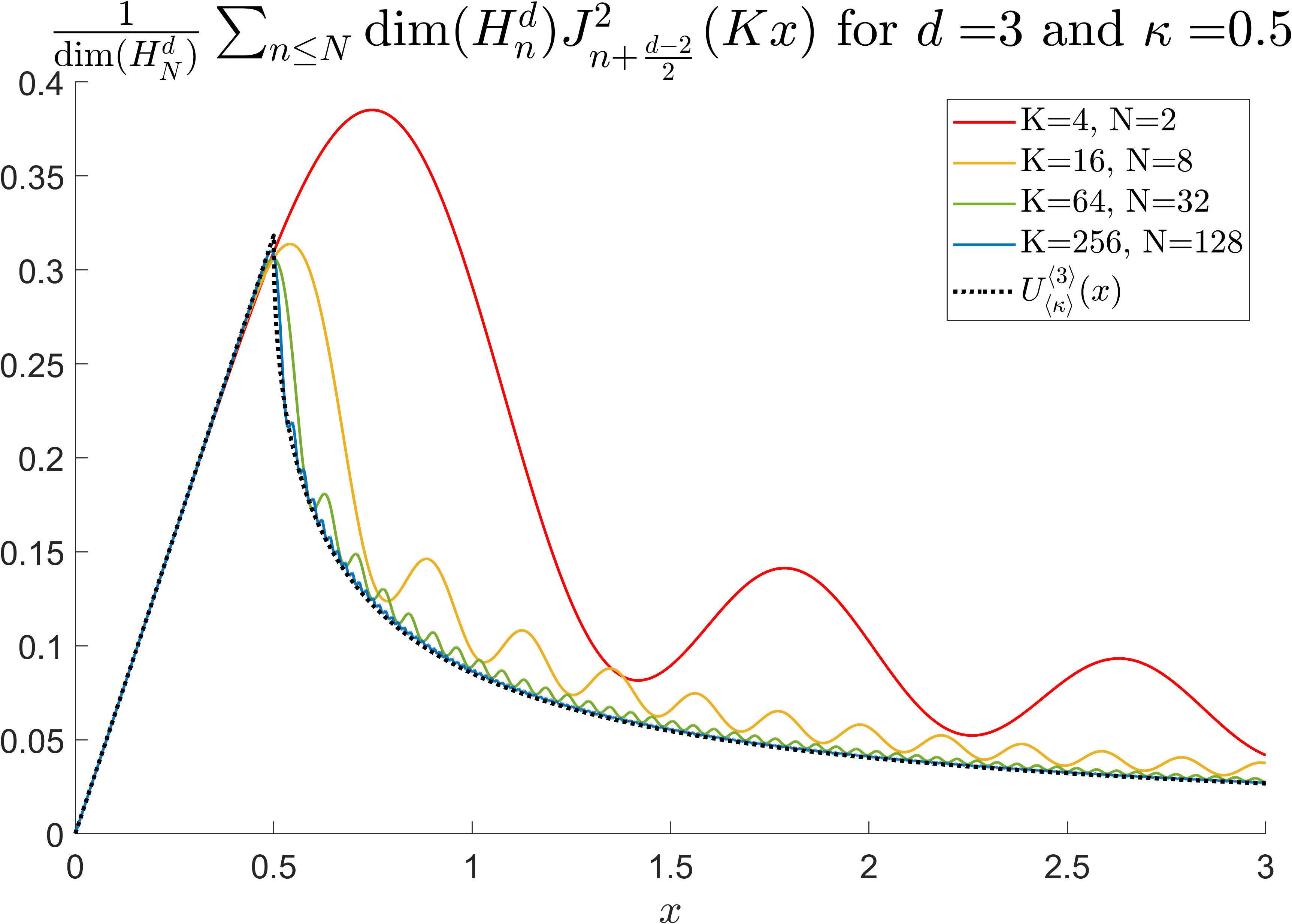}
\caption{Bessel sum, $\kappa=0.5$.}
\end{subfigure}
\hfill
\begin{subfigure}[t]{0.48\textwidth}
\centering
\includegraphics[width=\linewidth]{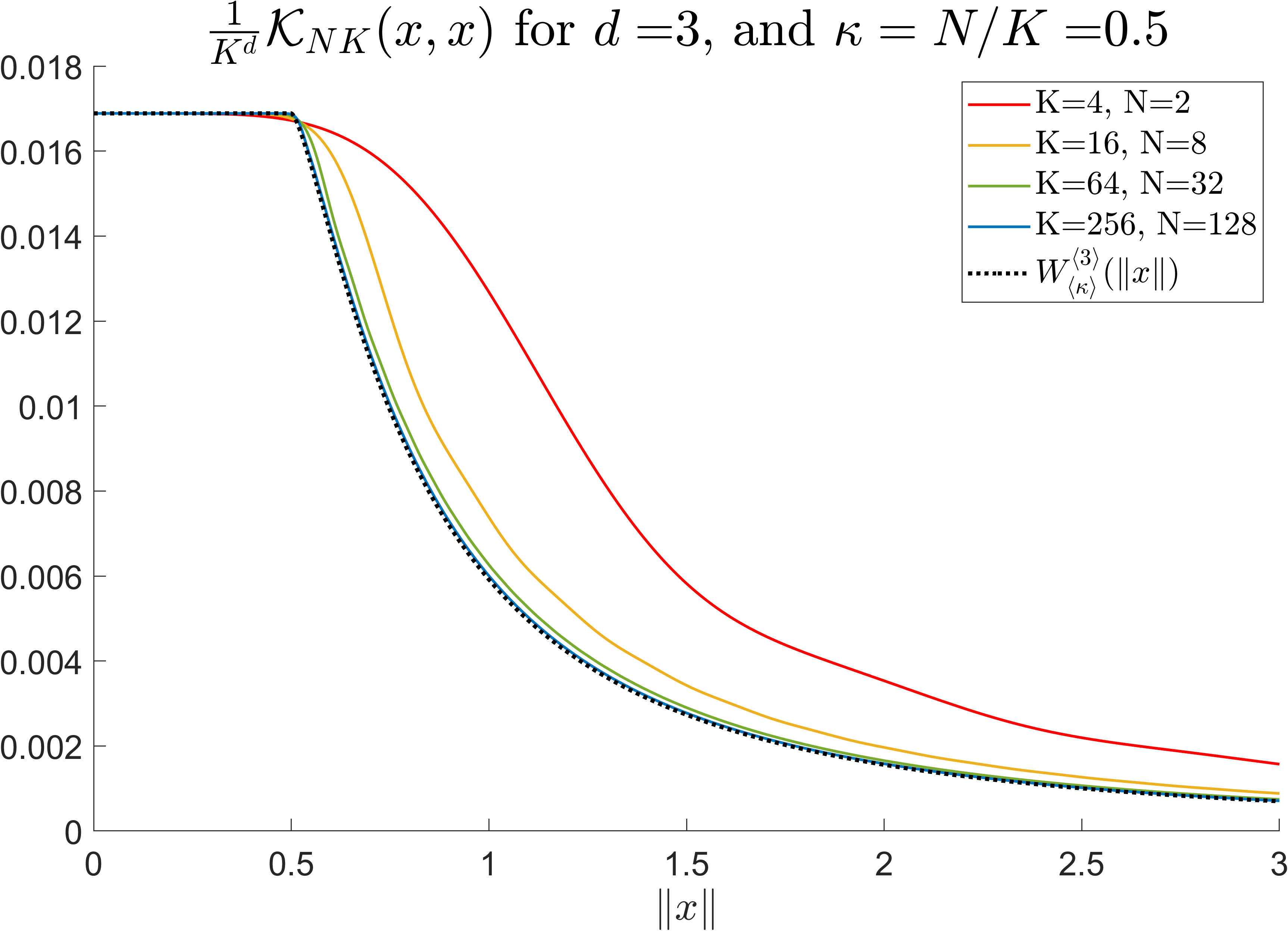}
\caption{Normalized diagonal kernel, $\kappa=0.5$.}
\end{subfigure}
\caption{Numerical convergence in dimension $d=3$, analogous to Figure~\ref{fig:DemoForD2}. The left column illustrates the finite Bessel sums in~\eqref{eqnC5:summation2integral2}, and the right column illustrates the normalized diagonal kernel in~\eqref{eqnC5:convergenceofshbKtoW}. The top and bottom rows correspond to $\kappa=1$ and $\kappa=0.5$, respectively; the dotted curves are the limiting profiles.}
\label{fig:DemoForD3}
\end{figure}
\FloatBarrier

\section{Spatiospectral concentration for the SFB truncation spaces}\label{secC5:SlepianProblem}

The SFB Slepian functions studied computationally in the three-dimensional setting in \cite{Khalid2016a} arise from eigenvalue problems associated with spatial and spectral projections. While that work demonstrated the usefulness of these functions computationally, the \emph{asymptotic} eigenvalue distribution and the corresponding Shannon-number law were not formally analyzed there. For standard Fourier spatio-spectral limiting operators, higher-dimensional eigenvalue-distribution estimates have been developed (see, e.g., \cite{IsraelMayeli2024,HughesIsraelMayeli2025}). However, in the present setting the spectral projection is the Paley-Wiener projection with an additional spherical-harmonic cutoff, which is not covered by the existing analysis. In particular, here the leading Shannon number is determined by the radial profile of the diagonal reproducing kernel rather than by a translation-invariant phase-space volume. In the setting considered here, the spectral projection is
\begin{align*}
\bprojectionInCFive_{\degAInCFive,K}:=\mathrm{proj}_{\spaceInCFive_{\degAInCFive,K}},
\end{align*}
with reproducing kernel $\KLK$, and the spatial projection associated with a bounded domain $D\subset\R^d$ is
\begin{align*}
\sprojection_D f:=\chi_D f.
\end{align*}
We consider the two concentration operators
\begin{align*}
\sprojection_D\bprojectionInCFive_{\degAInCFive,K}\sprojection_D
\qquad\text{and}\qquad
\bprojectionInCFive_{\degAInCFive,K}\sprojection_D\bprojectionInCFive_{\degAInCFive,K}.
\end{align*}
They are self-adjoint positive operators in $L^2(\R^d)$ and, when $D$ has finite measure, they are compact. Moreover, the two operators have the same non-zero eigenvalues. For $f\in \spaceInCFive_{\degAInCFive,K}$, the Rayleigh quotient of the second operator is
\begin{align*}
\frac{\langle \bprojectionInCFive_{\degAInCFive,K}\sprojection_D\bprojectionInCFive_{\degAInCFive,K}f,f\rangle_{L^2(\R^d)}}{\|f\|_{L^2(\R^d)}^2}
=
\frac{\int_D |f(x)|^2\diffsymbol x}{\int_{\R^d}|f(x)|^2\diffsymbol x}.
\end{align*}
Thus eigenvalues close to one correspond to SFB band-limited functions that are well concentrated in $D$, while eigenvalues close to zero correspond to functions whose energy is mostly outside $D$.

Define the Shannon number as the trace of the concentration operator:
\begin{align}\label{eq:trace-shannon-SFB}
\operatorname{Sh}_{\degAInCFive,K}(D)
&:=\operatorname{tr}(\sprojection_D\bprojectionInCFive_{\degAInCFive,K}\sprojection_D) \nonumber\\
&=\sum_i\lambdaInCFive_i(D;\degAInCFive,K)
=\int_D \KLK(x,x)\diffsymbol x.
\end{align}
Thus the Shannon number is an exact trace quantity, whereas its interpretation as the number of well-concentrated modes requires information about the distribution of the eigenvalues.

By the diagonal-kernel asymptotics of Section~\ref{secC5:sec2}, define the domain-dependent Shannon coefficient
\begin{align}\label{eq:radial-shannon-measure}
\mu_\kappa(D):=\int_D {\functionInCFiveFB}^{\langle d\rangle}_{\langle\kappa\rangle}(\|x\|)\diffsymbol x.
\end{align}
As $\degAInCFive,K\to\infty$ with $\degAInCFive/K\to\kappa$,
\[
\operatorname{Sh}_{\degAInCFive,K}(D)
=K^d\mu_\kappa(D)+\smallo(K^d).
\]
At the Paley--Wiener endpoint ($\kappa=\infty$), the density is constant, so $\mu_\infty(D)$ is proportional to the Euclidean volume of $D$. For finite $\kappa$, it is obtained by integrating a radial, non-translation-invariant density; hence the leading Shannon number may depend on the radial placement of $D$ as well as on its volume.

The trace asymptotic alone controls only the sum of the eigenvalues. The theorem below gives the stronger spectral statement: the number of near-one eigenvalues has the same leading term, while the number of eigenvalues in the intermediate transition region is of lower order. In this sense, the Shannon-number asymptotic predicts the effective dimension of the concentration problem.

\begin{thm}\label{thm:eigendistri}
Let $D\subset\R^d$ be a bounded Lipschitz domain, let $0<\kappa<\infty$, and let $\{\degAInCFive_j\}_{j\in\N}\subset\N_0$ and $\{K_j\}_{j\in\N}\subset\R_+$ satisfy $\degAInCFive_j\to\infty$, $K_j\to\infty$, and $\degAInCFive_j/K_j\to\kappa$. For each $j$, let $\{\lambdaInCFive_i(D;\degAInCFive_j,K_j)\}_{i=1}^{\infty}$ be the eigenvalues of $\sprojection_D\bprojectionInCFive_{\degAInCFive_j,K_j}\sprojection_D$, arranged in decreasing order. Then, for every $\epsilon\in(0,1/2)$,
\begin{align}\label{eqnC5:eigendistri1}
\lim_{j\to\infty}\frac{\sharp\{i:\lambdaInCFive_i(D;\degAInCFive_j,K_j)\geq 1-\epsilon\}}{K_j^{d}} = \int_{D} {\functionInCFiveFB}^{\langle d \rangle}_{\langle \kappa \rangle}(\|x\|)\diffsymbol x,
\end{align}
and
\begin{align}\label{eqnC5:eigendistri2}
\lim_{j\to\infty}\frac{\sharp\{i:\epsilon<\lambdaInCFive_i(D;\degAInCFive_j,K_j)< 1-\epsilon\}}{K_j^{d}} = 0.
\end{align} 
Here ${\functionInCFiveFB}^{\langle d \rangle}_{\langle \kappa \rangle}(\cdot)$ is the function given in Definitions \ref{def:UandW} and \ref{def:dilation}.
\end{thm}

\begin{rem}The endpoint cases $\kappa=0$ and $\kappa=\infty$ follow by monotonicity, as in Remark~\ref{rem:extension_analysisdiagonalKDimd}, using
$\sprojection_{D}\bprojectionInCFive_{\degAInCFive_1, K}\sprojection_{D} \preceq \sprojection_{D}\bprojectionInCFive_{\degAInCFive_2, K}\sprojection_{D}\preceq \sprojection_{D}\mathrm{proj}_{\mathrm{PW}_{K}}\sprojection_{D}$
whenever $\degAInCFive_1\leq \degAInCFive_2$, where $\preceq$ denotes the Loewner order.
\end{rem}

\begin{rem}[Domain-dependent Shannon coefficient]\label{rem:radial-shannon-measure}
With the notation \eqref{eq:radial-shannon-measure}, Theorem~\ref{thm:eigendistri} says that, along any sequence with $\degAInCFive/K\to\kappa$,
\begin{align*}
\sharp\{i:\lambdaInCFive_i(D;\degAInCFive,K)\geq 1-\epsilon\}
\sim K^d\mu_\kappa(D),
\end{align*}
up to a lower-order transition region. This is the SFB analogue of the classical space-bandwidth law: $K^d$ supplies the bandwidth scaling, $\mu_\kappa(D)$ is the domain-dependent Shannon coefficient, and $K^d\mu_\kappa(D)$ is the leading Shannon number, or leading effective dimension. In the endpoint $\kappa=\infty$, corresponding to the Paley--Wiener projection, this coefficient reduces to a constant multiple of the Lebesgue measure of $D$. For finite $\kappa$, however, the coefficient depends not only on the volume of $D$, but also on its radial placement relative to the transition scale $\kappa$. Consequently, two domains with the same volume may have different leading Shannon numbers.
\end{rem}
	
The rest of this section is devoted to the proof of Theorem \ref{thm:eigendistri}. The crucial step is to describe how the energy of $\KLK(x,y)$ concentrates near the diagonal $x=y$; this is done in Proposition \ref{prop:concentrationofK\degAInCFive, K}. Before proceeding to the main proof, we need the following property of spherical harmonic projections.

\begin{lem}\label{lem:concentrationofSH}
For any $\varphi>0$ and $0<\epsilon<1$, there exists a constant $\Phi$, depending on $\varphi$ and $\epsilon$, such that every $f\in L^2(\Sphered)$ supported on the spherical cap $\mathcal{C}(x_0,\frac{\varphi}{\degAInCFive})$ satisfies, for all sufficiently large $\degAInCFive$,
\begin{align}
\|\mathrm{proj}_{\mathrm{Harm}_{\degAInCFive}(\Sphered)} f\|^{2}_{L^2(\mathcal{C}^{\complement}(x_0,\frac{\Phi}{\degAInCFive}))} \leq \epsilon \|f\|^{2}_{L^2(\Sphered)},
\end{align}
where $\mathcal{C}(x,\varphi):=\{y\in\Sphered: \arccos(\langle x, y\rangle) < \varphi\}$ and $\mathcal{C}^{\complement}(x,\varphi):=\Sphered\setminus\mathcal{C}(x,\varphi)$.
\end{lem}
\begin{proof} We use the integral representation of $\mathrm{proj}_{\mathrm{Harm}_{\degAInCFive}(\Sphered)}$ (see, e.g., \cite[Prop. 2.2.1]{DaiXu}):
\begin{align}
\mathrm{proj}_{\mathrm{Harm}_{\degAInCFive}(\Sphered)}f(y)=\int_{\Sphered} f(x) \mathcal{K}_{\mathrm{Harm_{\degAInCFive}}}(x,y)\diffsymbol \omega(x),
\end{align}
where $\mathcal{K}_{\mathrm{Harm_{\degAInCFive}}}$ is the reproducing kernel of $\mathrm{Harm}_{\degAInCFive}(\Sphered)$, which has the form
\begin{align}
\mathcal{K}_{\mathrm{Harm_{\degAInCFive}}}(x,y)=\frac{(d-1)_{\degAInCFive}}{(\frac{d-1}{2})_{\degAInCFive}}P_{\degAInCFive}^{\frac{d-1}{2},\frac{d-3}{2}}(\langle x,y\rangle)\simeq \degAInCFive^{\frac{d-1}{2}}P_{\degAInCFive}^{\frac{d-1}{2},\frac{d-3}{2}}(\langle x,y\rangle).
\end{align}
Here $(\cdot)_{\degAInCFive}$ denotes the Pochhammer symbol, and $P_{\degAInCFive}^{\frac{d-1}{2},\frac{d-3}{2}}$ denotes the corresponding Jacobi polynomial.
Using the Cauchy-Schwarz inequality and the localized support of $f$, we obtain the estimate
\begin{align}
|\mathrm{proj}_{\mathrm{Harm}_{\degAInCFive}(\Sphered)}f(y)|^2\leq \int_{\mathcal{C}(x_0,\frac{\varphi}{\degAInCFive})}\left| \mathcal{K}_{\mathrm{Harm_{\degAInCFive}}}(x,y)\right|^2\diffsymbol \omega(x)\|f\|^2_{L^2(\Sphered)},
\end{align}
which yields
\begin{align}\label{eqnC5:CSInq}
&\int_{\mathcal{C}^{\complement}(x_0,\frac{\Phi}{\degAInCFive})}|\mathrm{proj}_{\mathrm{Harm}_{\degAInCFive}(\Sphered)}f(y)|^2 \diffsymbol \omega(y)\\
\nonumber		&\leq \left(\int_{\mathcal{C}^{\complement}(x_0,\frac{\Phi}{\degAInCFive})}\int_{\mathcal{C}(x_0,\frac{\varphi}{\degAInCFive})}\left|\mathcal{K}_{\mathrm{Harm_{\degAInCFive}}}(x,y)\right|^2\diffsymbol \omega(x) \diffsymbol \omega(y)\right)\|f\|^2_{L^2(\Sphered)}.
\end{align}
To bound the quantity in parentheses, we follow \cite{Marzo2007} and apply Szegő's estimate:
\begin{align}\label{eqnC5:estimateofJacobiP1}
P^{\lambda+1,\lambda}_n(\cos \theta)=\frac{k(\theta)}{\sqrt{n}}\left(\cos\left((n+\lambda+1)\theta-\frac{\pi}{2}\left(\lambda+\frac{3}{2}\right)\right)+\frac{\mathcal{O}(1)}{n\sin\theta}\right)\textnormal{ for }n\to\infty,
\end{align}
where $k(\theta)=\pi^{-1/2}(\sin\frac{\theta}{2})^{-\lambda-3/2}(\cos\frac{\theta}{2})^{-\lambda-1/2}$. This estimate holds for fixed $c>0$ and $c/L<\theta<\pi-c/L$.

For the right-hand side of~\eqref{eqnC5:CSInq},
\begin{align}
&	\int_{\mathcal{C}^{\complement}(x_0,\frac{\Phi}{\degAInCFive})}\int_{\mathcal{C}(x_0,\frac{\varphi}{\degAInCFive})}\left|\mathcal{K}_{\mathrm{Harm_{\degAInCFive}}}(x,y)\right|^2\diffsymbol \omega(x) \diffsymbol \omega(y) \\
\nonumber	&\leq \mathrm{vol}(\mathcal{C}(x_0,\frac{\varphi}{\degAInCFive})) \int_{\mathcal{C}^{\complement}(x_0,\frac{\Phi-\varphi}{\degAInCFive})}\left| \mathcal{K}_{\mathrm{Harm_{\degAInCFive}}}(x_0,y)\right|^2\diffsymbol \omega(y)\\
\nonumber		&\lesssim (\frac{\varphi}{\degAInCFive})^{d-1} 
{\degAInCFive}^{d-1}\int_{\frac{\Phi-\varphi}{\degAInCFive}}^{\pi} \left|P_{\degAInCFive}^{\frac{d-1}{2},\frac{d-3}{2}}(\cos \theta)\right|^2\sin^{d-2}\theta \diffsymbol \theta\\
\nonumber		&\lesssim \varphi^{d-1} \int_{\frac{\Phi-\varphi}{\degAInCFive}}^{\pi}\left| P_{\degAInCFive}^{\frac{d-1}{2},\frac{d-3}{2}}(\cos \theta)\right|^2\sin^{d-2}\theta \diffsymbol \theta.
\end{align}
We split the interval $(\frac{\Phi-\varphi}{\degAInCFive},\pi)$ into $(\frac{\Phi-\varphi}{\degAInCFive},\frac{\pi}{2})$, $(\frac{\pi}{2},\pi-\frac{\Phi-\varphi}{\degAInCFive})$, and $(\pi-\frac{\Phi-\varphi}{\degAInCFive},\pi)$. The integral of $|P_{\degAInCFive}^{\frac{d-1}{2},\frac{d-3}{2}}(\cos\theta)|^2\sin^{d-2}\theta$ over the last two intervals is $\bigo(\degAInCFive^{-1})$ (see, e.g., \cite[App.~C, proof of (4.17), particularly (C.14)--(C.15)]{gerhards2023slepian}). On the first interval,
\begin{align}
\nonumber		\int_{\frac{\Phi-\varphi}{\degAInCFive}}^{\pi/2} \left|P_{\degAInCFive}^{\frac{d-1}{2},\frac{d-3}{2}}(\cos \theta)\right|^2\sin^{d-2}\theta \diffsymbol \theta
&\lesssim \int_{\frac{\Phi-\varphi}{\degAInCFive}}^{\pi/2} \frac{k^2(\theta)}{\degAInCFive} \sin^{d-2}\theta \diffsymbol \theta\simeq  \int_{\frac{\Phi-\varphi}{\degAInCFive}}^{\pi/2} \frac{2^{d-2}}{\degAInCFive \sin^2\frac{\theta}{2}}  \diffsymbol \theta\\
&\lesssim \frac{1}{\degAInCFive}\cot\frac{\Phi-\varphi}{2\degAInCFive}\simeq (\Phi-\varphi)^{-1}.
\end{align}
Consequently,
\begin{align}
\nonumber \frac{	\|\mathrm{proj}_{\mathrm{Harm}_{\degAInCFive}(\Sphered)} f\|^{2}_{L^2(\mathcal{C}^{\complement}(x_0,\frac{\Phi}{\degAInCFive}))} }{\|f\|^{2}_{L^2(\Sphered)}}&\leq	\int_{\mathcal{C}^{\complement}(x_0,\frac{\Phi}{\degAInCFive})}\int_{\mathcal{C}(x_0,\frac{\varphi}{\degAInCFive})}\left|\mathcal{K}_{\mathrm{Harm_{\degAInCFive}}}(x,y)\right|^2\diffsymbol \omega(x) \diffsymbol \omega(y) \\
&\leq c\frac{\varphi^{d-1}}{\Phi-\varphi} +\bigo({\degAInCFive}^{-1}),
\end{align}
for a positive constant $c$ that is independent of $\degAInCFive,\Phi,\varphi$. Choosing $\Phi-\varphi$ sufficiently large, of order $\epsilon^{-1}\varphi^{d-1}$, gives the desired estimate.
\end{proof}

\begin{prop}\label{prop:concentrationofK\degAInCFive, K}
Let $r_0>0$ and $0<\kappa<\infty$. For every $\varepsilon>0$, there exist $\Delta,\Phi>0$ such that, for every $x\in\R^d$ with $\|x\|>r_0$ and all sufficiently large $\degAInCFive,K$ satisfying $\degAInCFive/K=\kappa$,
\begin{align}\label{eqnC5:concentrationKLK}
\|\KLK (x,\cdot)\|^{2}_{L^2(\mathcal{U}^{\complement}(r_x\xi_x;\frac{\Delta}{K},\frac{\Phi}{\degAInCFive}))} \leq \varepsilon K^{d},
\end{align}
where, for $r_x,\delta,\theta\geq0$ and $\xi_x\in\Sphered$, $\mathcal{U}(r_x\xi_x;\delta,\theta)$ denotes the sharp tessroid neighborhood of $x=r_x\xi_x$:
\begin{align*}
\mathcal{U}(r_x\xi_x;\delta,\theta):=\{r\xi\in\R^d: |r-r_x| < \delta, \xi\in\mathcal{C}(\xi_x,\theta) \},
\end{align*}
and $\mathcal{U}^{\complement}(r_x\xi_x;\delta,\theta)$ denotes its complement in $\R^d$.
\end{prop}
\begin{proof}
We use the closed-form expression for the reproducing kernel of $\mathrm{PW}_{K}$ (see, e.g., \cite[Lem. 12.2]{wendland05}):
\begin{align}\label{eqnC5:explicitKPWK}
\mathcal{K}_{\mathrm{PW}_{K}}(x,y)= \mathcal{F}^{-1} \left(\frac{\chi_{\BBCFour{K}}}{(2\pi)^{d/2}}\right)(x-y)=\frac{K^d }{(2\pi)^{d/2}} \frac{J_{d/2}(K\|x-y\|)}{(K\|x-y\|)^{d/2}}.
\end{align}
For fixed $x$, the restriction of $\KLK(x,\cdot)$ to the sphere of radius $r$ is the projection of $\mathcal{K}_{\mathrm{PW}_{K}}(x,r\,\cdot)$ onto $\mathrm{Harm}_{\degAInCFive}(\Sphered)$; that is,
\begin{align}
\KLK(x,r_y\xi_y) = \mathrm{proj}_{\mathrm{Harm}_{\degAInCFive}(\Sphered)} \{\mathcal{K}_{\mathrm{PW}_{K}}(x, r_y\cdot)\}(\xi_y).
\end{align}
This can be derived either from the Gegenbauer addition theorem for Bessel functions (cf. \cite[Chap. 11.42, Eq. (17)]{Watson}) or by treating $\spaceInCFive_{\degAInCFive,K}$ as the spherical harmonic truncation of $\mathrm{PW}_{K}$.

As follows directly from~\eqref{eqnC5:explicitKPWK}, $\mathcal{K}_{\mathrm{PW}_{K}}(x,\cdot)$ is obtained from $\mathcal{K}_{\mathrm{PW}_{1}}(0,\cdot)$ by dilation and translation. Set $E^{\langle d\rangle}:=\|\mathcal{K}_{\mathrm{PW}_{1}}(0,\cdot)\|^2_{L^2(\R^d)}$. Then $\|\mathcal{K}_{\mathrm{PW}_{K}}(x,\cdot)\|^2_{L^2(\R^d)}=K^dE^{\langle d\rangle}$. Moreover, for every $\varepsilon>0$, there exists $\Delta>0$, independent of $K$ and $x$, such that
\begin{align}\label{eqnC5:concentrationofKPWK}
\frac{\|\mathcal{K}_{\mathrm{PW}_{K}}(x,\cdot)\|^{2}_{L^2(\BBCFour{\nicefrac{\Delta}{K}}(x))}}{\|\mathcal{K}_{\mathrm{PW}_{K}}(x,\cdot)\|^{2}_{L^2(\R^d)}} \geq 1- \frac{\varepsilon}{4  E^{\langle d \rangle}},
\end{align}
where $\B^{d}_{\nicefrac{\Delta}{K}}(x)$ denotes the closed ball centered at $x$ with radius $\nicefrac{\Delta}{K}$. For fixed $x$, define
\begin{align*}
F_1(\cdot)=\mathcal{K}_{\mathrm{PW}_{K}}(x,\cdot)\chi_{\BBCFour{\nicefrac{\Delta}{K}}(x)}&&\text{and}&& F_2(\cdot)=\mathcal{K}_{\mathrm{PW}_{K}}(x,\cdot)-F_1(\cdot),
\end{align*}
Then~\eqref{eqnC5:concentrationofKPWK} gives $\|F_1\|^{2}_{L^2(\R^d)}\leq K^dE^{\langle d\rangle}$ and $\|F_2\|^{2}_{L^2(\R^d)}\leq\frac{\varepsilon}{4}K^d$. Accordingly, we decompose $\KLK(x,\cdot)$ as
\begin{align}
\nonumber	\KLK(x,r_y\xi_y)& = \mathrm{proj}_{\mathrm{Harm}_{\degAInCFive}(\Sphered)} \{F_1( r_y\cdot)\}(\xi_y) + \mathrm{proj}_{\mathrm{Harm}_{\degAInCFive}(\Sphered)} \{F_2( r_y\cdot)\}(\xi_y) \\
&=:\mathcal{K}_1(r_y\xi_y)+\mathcal{K}_2(r_y\xi_y).
\end{align}
Since orthogonal projection decreases the $L^2$ norm, we can control the $L^2$ norm of $\mathcal{K}_2$ by $F_2$ as follows:
\begin{align}\label{eqnC5:estimationK1}
\|\mathcal{K}_2\|_{L^2(\R^d)}^2&=\int_{0}^{\infty}r^{d-1}\int_{\Sphered}\left|\mathrm{proj}_{\mathrm{Harm}_{\degAInCFive}(\Sphered)} \{F_2( r\cdot)\}(\xi)\right|^2 \diffsymbol \omega(\xi) \diffsymbol r\\
\nonumber	&\leq\int_{0}^{\infty}r^{d-1}\int_{\Sphered}\left|\{F_2( r\cdot)\}(\xi)\right|^2 \diffsymbol \omega(\xi) \diffsymbol r \leq \|F_2\|_{L^2(\R^d)}^2\leq \frac{\varepsilon}{4}K^d.
\end{align}

For sufficiently large $K$, the support of $F_1$ satisfies
\[
\operatorname{ess\,supp}F_1
\subset
\mathcal U\left(
r_x\xi_x;\frac{\Delta}{K},
\frac{\varphi_*}{\degAInCFive}
\right),
\qquad
\varphi_*:=\frac{2\kappa\Delta}{r_0},
\]
uniformly for $\|x\|\geq r_0$. Indeed,
\[
\degAInCFive\arcsin\frac{\Delta}{K\|x\|}
\leq\varphi_*
\]
for all sufficiently large $\degAInCFive$. Lemma~\ref{lem:concentrationofSH}, applied with the fixed parameter $\varphi_*$ and error $\varepsilon/(4E^{\langle d\rangle})$ (after decreasing the error if necessary), therefore yields a constant $\Phi$ independent of $x$, $\degAInCFive$, and $K$.

\small
\begin{align}\label{eqnC5:estimationK2}
\|\mathcal{K}_1\|_{L^2(\mathcal{U}^{\complement}(r_x\xi_x;\frac{\Delta}{K},\frac{\Phi}{\degAInCFive}))}^2
&=
\int_{r_x-\Delta/K}^{r_x+\Delta/K}
r^{d-1}
\int_{\mathcal{C}^{\complement}(\xi_x,\frac{\Phi}{\degAInCFive})}
\left|
\mathrm{proj}_{\mathrm{Harm}_{\degAInCFive}(\Sphered)}
\{F_1(r\cdot)\}(\xi)
\right|^2
\diffsymbol\omega(\xi)\diffsymbol r\\ \nonumber
&\leq
\frac{\varepsilon}{4E^{\langle d\rangle}}
\int_{r_x-\Delta/K}^{r_x+\Delta/K}
r^{d-1}\|F_1(r\cdot)\|_{L^2(\Sphered)}^2
\diffsymbol r
\leq
\frac{\varepsilon}{4E^{\langle d\rangle}}
\|F_1\|_{L^2(\R^d)}^2
\leq
\frac{\varepsilon}{4}K^d.
\end{align}

Together with~\eqref{eqnC5:estimationK1}, this gives
\[
\|\KLK(x,\cdot)\|_{L^2(\mathcal U^\complement)}^2
\leq
2\|\mathcal K_1\|_{L^2(\mathcal U^\complement)}^2
+2\|\mathcal K_2\|_{L^2(\mathcal U^\complement)}^2
\leq\varepsilon K^d.
\]
This proves~\eqref{eqnC5:concentrationKLK}. The constants $\Delta$ and $\Phi$ depend only on $\kappa$, $\varepsilon$, $r_0$, and $d$, so the estimate is uniform for $\|x\|\geq r_0$.
\end{proof}
\begin{prop} Let $D$, $\kappa$, $\{\degAInCFive_j\}$, and $\{K_j\}$ satisfy the assumptions of Theorem~\ref{thm:eigendistri}, and abbreviate $\lambdaInCFive_i(D;\degAInCFive_j,K_j)$ by $\lambdaInCFive_i^{(j)}$. Then
\begin{align}\label{eqnC5:trace}
\lim_{j\to\infty} \frac{\sum_{i=1}^{\infty} \lambdaInCFive_i^{(j)}}{K_j^d}= \int_{D} {\functionInCFiveFB}^{\langle d \rangle}_{\langle \kappa \rangle}(\|x\|)\diffsymbol x,
\end{align}
and
\begin{align}\label{eqnC5:HSnorm}
\lim_{j\to\infty} \frac{\sum_{i=1}^{\infty} (\lambdaInCFive_i^{(j)})^{2}}{K_j^d}= \int_{D} {\functionInCFiveFB}^{\langle d \rangle}_{\langle \kappa \rangle}(\|x\|)\diffsymbol x.
\end{align}
\end{prop}

\begin{proof} We first treat the exact relation $\degAInCFive/K=\kappa$. We write $\sprojection\bprojectionInCFive\sprojection$ as an integral operator:
\begin{align}
\sprojection_{D}\bprojectionInCFive_{\degAInCFive, K}\sprojection_{D} f(x) &=\chi_{D}(x) \int_{\R^d} f(y)\chi_{D}(y)\KLK(x,y)\diffsymbol y\\ 
\nonumber		&=\int_{\R^d}\left[\chi_{D}(x)\chi_{D}(y) \KLK (x,y)\right] f(y)\diffsymbol y
\end{align}
Thus the integral kernel of $\sprojection\bprojectionInCFive\sprojection$ is $\chi_D(x)\chi_D(y)\KLK(x,y)$.
This gives the following trace and Hilbert-Schmidt norm identities (cf. \cite[Chap. VI.6]{reed1981functional}):
\begin{align*}
	\mathrm{tr}(\sprojection\bprojectionInCFive\sprojection)=	\sum_{i=1}^{\infty} \lambdaInCFive_i = \int_{D}  \KLK(x,x)\diffsymbol x \,\text{,}\quad \|\sprojection\bprojectionInCFive\sprojection\|^2_{\HS}=\sum_{i=1}^{\infty} \lambdaInCFive_i^2 = \int_{D}\int_{D} | \KLK(x,y)|^2\diffsymbol x \diffsymbol y. 
\end{align*}
By Theorem \ref{thm:convergenceofshbKtoW}, we immediately obtain~\eqref{eqnC5:trace}. We can bound the right-hand side of~\eqref{eqnC5:HSnorm} by
\begin{align}\label{eqnC5:HSnormUpperES}
\int_{D}\int_{D} | \KLK(x,y)|^2\diffsymbol x \diffsymbol y &\leq \int_{D}\int_{\R^d} | \KLK(x,y)|^2\diffsymbol x \diffsymbol y \\
\nonumber &= \int_{D}  \KLK(x,x)\diffsymbol x = K^d \int_{D} {\functionInCFiveFB}^{\langle d \rangle}_{\langle \kappa \rangle}(\|x\|)\diffsymbol x +\smallo(K^d).
\end{align}
The core of the proof is the lower bound for $\int_{D}\int_{D} | \KLK(x,y)|^2\diffsymbol x \diffsymbol y$. Here we use the concentration property of $\KLK$ near $x=y$ from Proposition \ref{prop:concentrationofK\degAInCFive, K} as follows.

Fix $\varepsilon>0$ and decompose $D=D_i\cup D_b$, where $D_i$ and $D_b$ denote the interior and boundary regions, respectively:
\begin{align*}
D_i=\{x\in D: \|x\|>r_0, \mathcal{U}(r_x\xi_x;\frac{\Delta}{K},\frac{\Phi}{\degAInCFive})\subset D\} &&\text{and}&& D_b=D\setminus D_i.
\end{align*}
Here $r_0>0$ is fixed, and $\Delta,\Phi$ are given by Proposition \ref{prop:concentrationofK\degAInCFive, K}. Since $D$ is a Lipschitz domain, it satisfies an interior cone condition. Therefore, by \cite[Lem. 3.7]{wendland05}, we have $\mathrm{vol}(D_b)\leq \mathrm{vol}(\B^d_{r_0})+\bigo(K^{-1})$, and hence
\begin{align}
\int_{D}\int_{D^{\complement}} | \KLK(x,y)|^2\diffsymbol x \diffsymbol y&  \leq 	\int_{D_b}\left[\int_{D^{\complement}} | \KLK(x,y)|^2\diffsymbol x\right] \diffsymbol y+\int_{D_i}\left[\int_{D^{\complement}} | \KLK(x,y)|^2\diffsymbol x\right] \diffsymbol y\\
\nonumber		&\leq \mathrm{vol}(D_b) K^d+\mathrm{vol}(D_i) \varepsilon K^d,
\end{align}
This leads to
\begin{align}\label{eqnC5:HSnormLowerES}
\nonumber& \lim_{\degAInCFive,K\to \infty,\frac{\degAInCFive}{K}=\kappa}\frac{	\int_{D}\int_{D} | \KLK(x,y)|^2\diffsymbol x \diffsymbol y}{K^d}  \\ \nonumber
&=  \lim_{\degAInCFive,K\to \infty,\frac{\degAInCFive}{K}=\kappa}\frac{1}{K^d}\left[\int_{D} \KLK(x,x)\diffsymbol x  -\int_{D}\int_{D^{\complement}} | \KLK(x,y)|^2\diffsymbol x \diffsymbol y\right]  \\ \nonumber
&\geq   \lim_{K\to\infty}\frac{1}{K^d}\left[K^d \int_{D} {\functionInCFiveFB}^{\langle d \rangle}_{\langle \kappa \rangle}(\|x\|)\diffsymbol x  -\varepsilon \mathrm{vol}(D)K^d -\smallo(K^d)-\mathrm{vol}(\B^d_{r_0})K^d\right]\\ 
&\geq \int_{D} {\functionInCFiveFB}^{\langle d \rangle}_{\langle \kappa \rangle}(\|x\|)\diffsymbol x -\varepsilon \mathrm{vol}(D)-\mathrm{vol}(\B^d_{r_0}).
\end{align}
The desired equality~\eqref{eqnC5:HSnorm} then follows from~\eqref{eqnC5:HSnormUpperES} and~\eqref{eqnC5:HSnormLowerES} by letting $\varepsilon,r_0\to 0$.

The general case follows because the concentration estimate is uniform when $\degAInCFive/K$ remains in a compact subset of $(0,\infty)$, while the diagonal-kernel limit already holds for arbitrary sequences.
\end{proof}

\begin{proof}[Proof of Theorem \ref{thm:eigendistri}] Along the sequence in the theorem, we suppress the index $j$. We abbreviate $\lambdaInCFive_i(D;\degAInCFive,K)$ by $\lambdaInCFive_i$. The estimates~\eqref{eqnC5:trace} and~\eqref{eqnC5:HSnorm} give
\begin{align}
\Delta_{K}:=\sum_{i=1}^{\infty} \left(\lambdaInCFive_i-\lambdaInCFive_i^{2}\right)= \sum_{i=1}^{\infty} \lambdaInCFive_i-\sum_{i=1}^{\infty} \lambdaInCFive_i^2 =\smallo(K^d).
\end{align}
Since every $\lambdaInCFive_i$ lies in $[0,1]$, and $h(t):=t-t^2=t(1-t)$ is non-negative on $[0,1]$ with minimum $\epsilon(1-\epsilon)$ on $[\epsilon,1-\epsilon]$,
\begin{align}
\sharp\{\lambdaInCFive_i:\epsilon \leq \lambdaInCFive_i\leq 1-\epsilon\} \leq \frac{1}{\epsilon(1-\epsilon)}\sum_{i=1}^{\infty} \left(\lambdaInCFive_i-\lambdaInCFive_i^{2}\right) =\smallo(K^d),
\end{align}
which yields~\eqref{eqnC5:eigendistri2}.

To prove~\eqref{eqnC5:eigendistri1}, note that $h(t)\geq t/2$ for $0<t<1/2$. Hence
\begin{align}
\sum_{i\in \N;\lambdaInCFive_i\leq \nicefrac{1}{2}} \lambdaInCFive_i \leq 2 	\sum_{i=1}^{\infty} \left(\lambdaInCFive_i-\lambdaInCFive_i^{2}\right) =2\Delta_{K}=\smallo(K^d).
\end{align}
Choose $0<\delta<\epsilon<1/2$, and let $A_{K,\delta}$ and $B_{K,\delta}$ denote the numbers of eigenvalues in $(1-\delta,1]$ and $(1/2,1-\delta]$, respectively. Then
\begin{align}
(1-\delta) A_{K,\delta} + \frac{B_{K,\delta}}{2} +\sum_{i\in \N; \lambdaInCFive_i\leq 1/2}\lambdaInCFive_i \leq \sum_{i=1}^{\infty} \lambdaInCFive_i \leq A_{K,\delta} +B_{K,\delta}+\sum_{i\in \N; \lambdaInCFive_i\leq 1/2} \lambdaInCFive_i.
\end{align}
By the previous analysis, $B_{K,\delta}$ and $\sum_{i\in \N;\lambdaInCFive_i\leq 1/2}\lambdaInCFive_i$ are controlled by $2\delta^{-1}\Delta_{K}$ and $2\Delta_{K}$, respectively. Thus we have the chain of inequalities
\begin{align}\label{eqnC5:control_chain_trace}
\sum_{i=1}^{\infty}\lambdaInCFive_i- (2\delta^{-1}+2)\Delta_{K} \leq A_{K,\delta} &\leq \sharp\{\lambdaInCFive_i: 1-\epsilon <\lambdaInCFive_i\leq 1\} \\ \nonumber
&\leq A_{K,\delta}+B_{K,\delta} \leq \frac{1}{(1-\delta)} [\sum_{i=1}^{\infty} \lambdaInCFive_i +\delta^{-1}\Delta_{K}].
\end{align} 
Dividing all terms in~\eqref{eqnC5:control_chain_trace} by $K^d$ and choosing $\delta\to 0$ so that $\frac{\Delta_{K}}{\delta K^d}\to 0$ (which is possible because $\Delta_{K}=\smallo(K^d)$) proves~\eqref{eqnC5:eigendistri1}.

\end{proof}

\section{Radial placement and SFB concentration spectra}\label{secC5:radial-nystrom}

The numerical experiments in this section are intended only to illustrate the asymptotic eigenvalue distribution and the radial dependence of the Shannon coefficient. They are not used in the proofs of the main results. To illustrate the spectral consequence of Theorem~\ref{thm:eigendistri} within the full-space SFB model, we restrict to localization domains that depend only on the radial variable. Let $\mathcal{I}\subset\R_+$ be a bounded measurable set, for instance an interval $[a,b]$ or a finite union of such intervals, and set
\begin{align}
D_{\mathcal{I}}:=\{x\in\R^d:\|x\|\in\mathcal{I}\}.
\end{align}
We discretize the corresponding SFB concentration operator directly, using the truncation space $\spaceInCFive_{\degAInCFive,K}$ from Sections~\ref{secC5:sec2} and~\ref{secC5:SlepianProblem}. For domains of the form $D_{\mathcal I}$, the angular variables remain diagonal, and the SFB concentration problem reduces to a family of one-dimensional Hankel-type concentration operators, one for each spherical harmonic degree. The numerical goal is not to approximate individual Slepian eigenfunctions, but to compare the observed eigenvalue plunge with the leading Shannon-number prediction and to isolate the effect of radial placement.

In Section~\ref{secC5:SlepianProblem}, the Slepian eigenvalues are defined through the compact operator
\begin{align}
\sprojection_{D_{\mathcal{I}}}\bprojectionInCFive_{\degAInCFive,K}\sprojection_{D_{\mathcal{I}}}.
\end{align}
For computation it is more convenient to work on the band-limited side with
\begin{align}\label{eq:app-bandlimited-concentration}
\bprojectionInCFive_{\degAInCFive,K}\sprojection_{D_{\mathcal{I}}}\bprojectionInCFive_{\degAInCFive,K}:\spaceInCFive_{\degAInCFive,K}\to\spaceInCFive_{\degAInCFive,K}.
\end{align}
The two operators have the same non-zero eigenvalues. For $f\in\spaceInCFive_{\degAInCFive,K}$, the Rayleigh quotient of~\eqref{eq:app-bandlimited-concentration} is exactly the spatial concentration ratio
\begin{align}\label{eq:app-rayleigh-concentration}
\frac{\langle \bprojectionInCFive_{\degAInCFive,K}\sprojection_{D_{\mathcal{I}}}\bprojectionInCFive_{\degAInCFive,K}f,f\rangle_{L^2(\R^d)}}{\|f\|^2_{L^2(\R^d)}}
=
\frac{\int_{D_{\mathcal{I}}}|f(x)|^2\diffsymbol x}{\int_{\R^d}|f(x)|^2\diffsymbol x}.
\end{align}
Thus the eigenvalues measure the fraction of energy of the corresponding SFB band-limited eigenfunctions contained in $D_{\mathcal{I}}$.

For a fixed spherical harmonic degree $\degaInCFive$, put
\begin{align}
\nu_{\degaInCFive}:=\degaInCFive+\frac{d-2}{2}.
\end{align}
If $F\in L^2((0,K),k\diffsymbol k)$, then the radial SFB component associated with the angular mode $Y_{\degaInCFive,\degb}$ can be written as
\begin{align}\label{eq:app-hankel-parametrization}
 g_F(r)=\int_0^K F(k)J_{\nu_{\degaInCFive}}(kr)k\diffsymbol k, \qquad
 f_F(r\xi)=r^{\frac{2-d}{2}}g_F(r)Y_{\degaInCFive,\degb}(\xi),
\end{align}
where $Y_{\degaInCFive,\degb}$ is normalized in $L^2(\Sphered)$. With the Hankel-transform normalization used in Section~\ref{sec:preliminary}, this gives a unitary correspondence on each fixed angular mode:
\begin{align}\label{eq:app-unitary-norm}
\|f_F\|^2_{L^2(\R^d)}
=
\int_0^{\infty}|g_F(r)|^2r\diffsymbol r
=
\int_0^K |F(k)|^2 k\diffsymbol k.
\end{align}
Here the factor $r^{(2-d)/2}$ in~\eqref{eq:app-hankel-parametrization} converts the $d$-dimensional volume element $r^{d-1}\diffsymbol r\diffsymbol\omega$ into the Hankel measure $r\diffsymbol r$ after angular integration.

Since $D_{\mathcal{I}}$ is radial, multiplication by $\chi_{D_{\mathcal{I}}}$ does not mix different spherical harmonic degrees or different basis elements within a fixed degree. Indeed,
\begin{align}
\int_{\Sphered}Y_{\degaInCFive,\degb}(\xi)\overline{Y_{\degaInCFive',\degb'}(\xi)}\diffsymbol\omega(\xi)
=\delta_{\degaInCFive\degaInCFive'}\delta_{\degb\degb'}.
\end{align}
Consequently, the operator~\eqref{eq:app-bandlimited-concentration} is block diagonal with respect to the angular indices $(\degaInCFive,\degb)$. For a fixed angular mode, the energy inside $D_{\mathcal{I}}$ is
\begin{align}\label{eq:app-local-energy}
\int_{D_{\mathcal{I}}}|f_F(x)|^2\diffsymbol x
&=\int_{\mathcal{I}}|g_F(r)|^2r\diffsymbol r \nonumber\\
&=\int_0^K\int_0^K
A_{\nu_{\degaInCFive},\mathcal{I}}(k,k')F(k')\overline{F(k)}\,k'\diffsymbol k'\,k\diffsymbol k,
\end{align}
where
\begin{align}\label{eq:app-radial-kernel}
A_{\nu,\mathcal{I}}(k,k'):=\int_{\mathcal{I}}J_{\nu}(kr)J_{\nu}(k'r)r\diffsymbol r.
\end{align}
Therefore, by the Riesz representation of the quadratic form~\eqref{eq:app-local-energy} in the spectral Hilbert space $L^2((0,K),k\diffsymbol k)$, the radial block of the concentration operator is the self-adjoint integral operator
\begin{align}\label{eq:app-hankel-block}
(T_{\nu_{\degaInCFive},K,\mathcal{I}}F)(k)
=
\int_0^K A_{\nu_{\degaInCFive},\mathcal{I}}(k,k')F(k')k'\diffsymbol k'.
\end{align}
Its Rayleigh quotient is precisely the concentration ratio on the angular mode $(\degaInCFive,\degb)$:
\begin{align}\label{eq:app-rayleigh-block}
\frac{\langle T_{\nu_{\degaInCFive},K,\mathcal{I}}F,F\rangle_{L^2((0,K),k\diffsymbol k)}}{\|F\|^2_{L^2((0,K),k\diffsymbol k)}}
=
\frac{\int_{D_{\mathcal{I}}}|f_F(x)|^2\diffsymbol x}{\int_{\R^d}|f_F(x)|^2\diffsymbol x}.
\end{align}
It follows that the eigenvalues of $T_{\nu_{\degaInCFive},K,\mathcal{I}}$ are the Slepian concentration eigenvalues on the fixed angular mode $(\degaInCFive,\degb)$. Since the same radial block occurs for every spherical harmonic of degree $\degaInCFive$, each such eigenvalue is counted with multiplicity $\mathrm{dim}(H_{\degaInCFive}^{d})$ in the full SFB concentration spectrum. Thus the non-zero spectrum of the full SFB concentration operator is obtained by taking the spectra of the blocks~\eqref{eq:app-hankel-block} for $0\leq\degaInCFive\leq\degAInCFive$, and repeating each eigenvalue with multiplicity $\mathrm{dim}(H_{\degaInCFive}^{d})$.

For numerical purposes, choose quadrature nodes $k_1,\ldots,k_m\in(0,K)$ and positive weights $q_1,\ldots,q_m$ for the measure $k\diffsymbol k$, so that
\begin{align}\label{eq:app-quadrature-k}
\int_0^K h(k)k\diffsymbol k\simeq \sum_{i=1}^m q_i h(k_i).
\end{align}
For example, here we take a Gauss--Legendre rule for the Lebesgue measure $\diffsymbol k$ on $[0,K]$ and absorb the factor $k_i$ into the weights. Discretizing the eigenvalue equation
\begin{align}
T_{\nu_{\degaInCFive},K,\mathcal{I}}F=\lambdaInCFive F
\end{align}
at the nodes $k_i$ gives the nodal system
\begin{align}\label{eq:app-nodal-eigenproblem}
\sum_{j=1}^m A_{\nu_{\degaInCFive},\mathcal{I}}(k_i,k_j)F(k_j)q_j
=\lambdaInCFive F(k_i),
\qquad 1\leq i\leq m.
\end{align}
This is the direct Nystr\"om discretization of~\eqref{eq:app-hankel-block}. However, because the underlying operator is self-adjoint with respect to the weighted inner product in $L^2((0,K),k\diffsymbol k)$, it is preferable to convert~\eqref{eq:app-nodal-eigenproblem} into an ordinary symmetric matrix eigenvalue problem. Setting
\begin{align}
u_i=\sqrt{q_i}F(k_i)
\end{align}
turns~\eqref{eq:app-nodal-eigenproblem} into
\begin{align}\label{eq:app-symmetric-eigenproblem}
\sum_{j=1}^m
\sqrt{q_i}\,A_{\nu_{\degaInCFive},\mathcal{I}}(k_i,k_j)\sqrt{q_j}\,u_j
=\lambdaInCFive u_i.
\end{align}
Accordingly, the symmetric Nystr\"om matrix for the $\degaInCFive$-th block is
\begin{align}\label{eq:app-nystrom-matrix}
M^{(\degaInCFive)}_{ij}
=
\sqrt{q_i}\,A_{\nu_{\degaInCFive},\mathcal{I}}(k_i,k_j)\sqrt{q_j},
\qquad 1\leq i,j\leq m.
\end{align}
The eigenvalues of $M^{(\degaInCFive)}$ approximate the eigenvalues of the integral operator $T_{\nu_{\degaInCFive},K,\mathcal{I}}$, and hence the Slepian concentration eigenvalues of the SFB problem on the angular block of degree $\degaInCFive$. Collecting these eigenvalues for $0\leq\degaInCFive\leq\degAInCFive$ and repeating them $\mathrm{dim}(H_{\degaInCFive}^{d})$ times gives a finite-dimensional approximation to the SFB Slepian eigenvalue distribution.

When $\mathcal{I}=[a,b]$, the radial kernel~\eqref{eq:app-radial-kernel} may be computed either by a one-dimensional quadrature in $r$ or by the standard Bessel integral formula
\begin{align}\label{eq:app-bessel-integral-offdiag}
A_{\nu,[a,b]}(s,t)=\left[\frac{r}{s^2-t^2}\big(sJ_{\nu+1}(sr)J_{\nu}(tr)-tJ_{\nu}(sr)J_{\nu+1}(tr)\big)\right]_{r=a}^{r=b},\qquad s\neq t,
\end{align}
and, on the diagonal,
\begin{align}\label{eq:app-bessel-integral-diag}
A_{\nu,[a,b]}(s,s)=\left[\frac{r^2}{2}\big(J_{\nu}(sr)^2-J_{\nu-1}(sr)J_{\nu+1}(sr)\big)\right]_{r=a}^{r=b}.
\end{align}
If an endpoint is $r=0$, the corresponding boundary term is understood in the limiting sense. For nearly equal nodes $s$ and $t$, it is numerically safer to use~\eqref{eq:app-bessel-integral-diag} or a direct quadrature in~\eqref{eq:app-radial-kernel}, rather than the off-diagonal formula~\eqref{eq:app-bessel-integral-offdiag}.

The discretization above is particularly useful for illustrating the radial-placement effect predicted by Theorem~\ref{thm:eigendistri}. For a radial localization set $D_{\mathcal{I}}$, the leading coefficient in Theorem~\ref{thm:eigendistri} is
\begin{align}\label{eq:app-radial-mu}
\mu_\kappa(D_{\mathcal{I}}):=\int_{D_{\mathcal{I}}}{\functionInCFiveFB}^{\langle d\rangle}_{\langle\kappa\rangle}(\|x\|)\diffsymbol x =\mathrm{vol}(\Sphered)\int_{\mathcal{I}}{\functionInCFiveFB}^{\langle d\rangle}_{\langle\kappa\rangle}(r)r^{d-1}\diffsymbol r.
\end{align}
Thus the predicted plunge location of the sorted Slepian eigenvalues is approximately
\begin{align}
K^d\mu_\kappa(D_{\mathcal{I}}),
\end{align}
up to lower-order terms. In the classical translation-invariant Paley--Wiener setting this coefficient is a constant multiple of $|D_{\mathcal{I}}|$, and therefore depends only on the Euclidean volume. In the present SFB truncation setting, however, \eqref{eq:app-radial-mu} depends on the radial placement of $\mathcal{I}$ through the non-constant density ${\functionInCFiveFB}^{\langle d\rangle}_{\langle\kappa\rangle}$.

\subsection{Numerical illustration of radial placement}\label{subsec:radial-placement-numerics}
We demonstrate this effect in dimension $d=2$. Let
\begin{align*}
\mathcal{I}_{1}=[0,0.8],\qquad
\mathcal{I}_{2}=[2,\sqrt{2^2+0.8^2}],\qquad
\mathcal{I}_{3}=[4,\sqrt{4^2+0.8^2}],
\end{align*}
and set $D_i=D_{\mathcal{I}_i}$, $i=1,2,3$. These three radial localization domains have the same area,
\begin{align*}
|D_1|=|D_2|=|D_3|=0.64\pi.
\end{align*}
We take $K=160$ and use the symmetric Nystr\"om discretization~\eqref{eq:app-nystrom-matrix} with $m=600$ quadrature nodes. A convergence test over $350\leq m\leq750$ showed that both the sorted spectra and the numerical traces had stabilized by $m=600$. For each $\kappa$, the eigenvalues from all angular blocks $0\leq \degaInCFive\leq \degAInCFive$ are collected with multiplicity $\mathrm{dim}(H_{\degaInCFive}^{2})$ and sorted in descending order. The vertical dashed lines in Figure~\ref{fig:eigen-spectra-d2-comparison} mark the predicted indices $K^2\mu_\kappa(D_i)$.

\begin{figure}[!p]
\centering
\begin{subfigure}{0.98\textwidth}
\centering
\includegraphics[width=\textwidth]{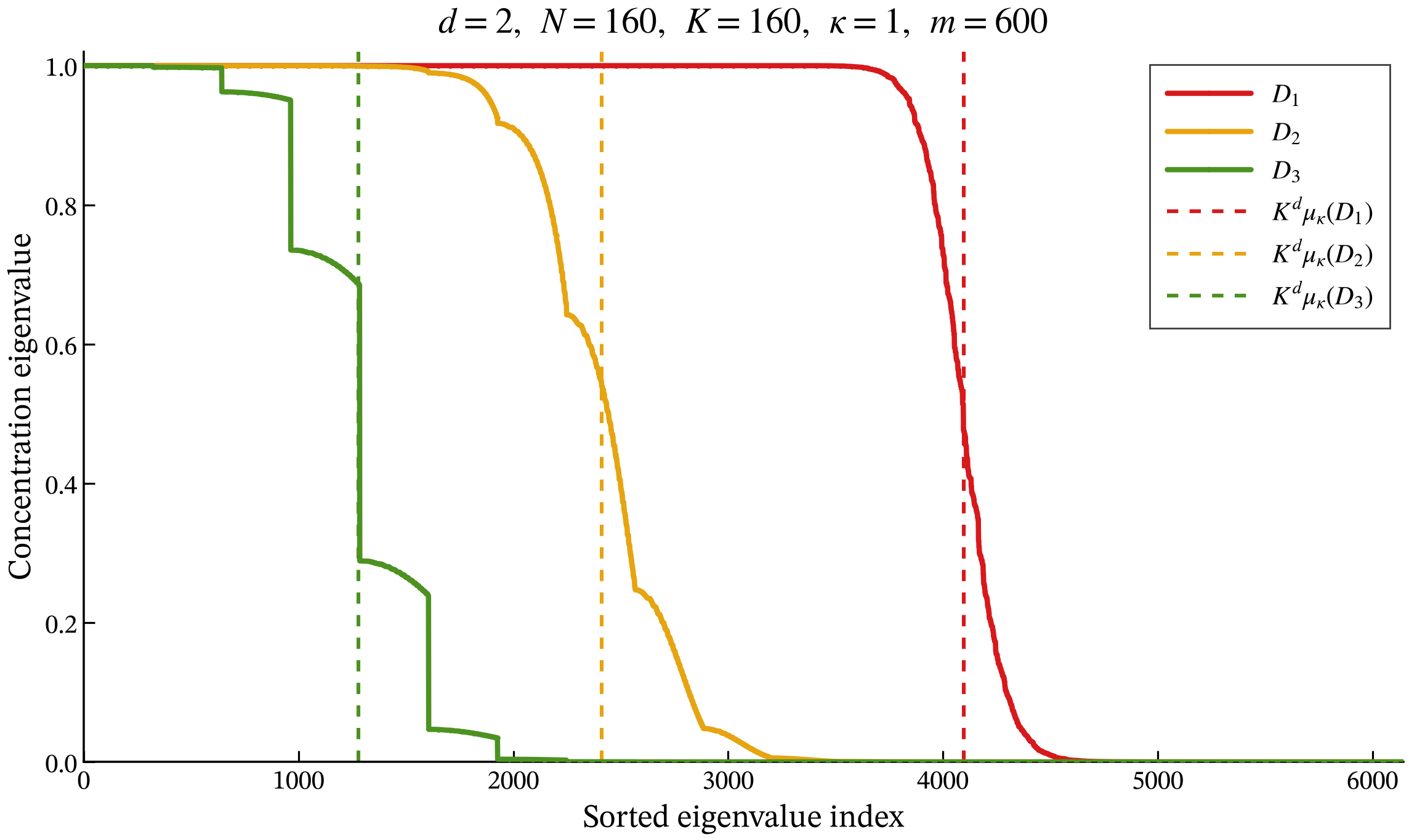}
\caption{$\kappa=1$ ($K=160$, $\degAInCFive=160$). The predicted indices $K^2\mu_\kappa(D_i)$ are $4096.000$, $2410.039$, and $1277.703$.}
\label{fig:eigen-spectra-d2-kappa1}
\end{subfigure}
\vspace{0.6em}
\begin{subfigure}{0.98\textwidth}
\centering
\includegraphics[width=\textwidth]{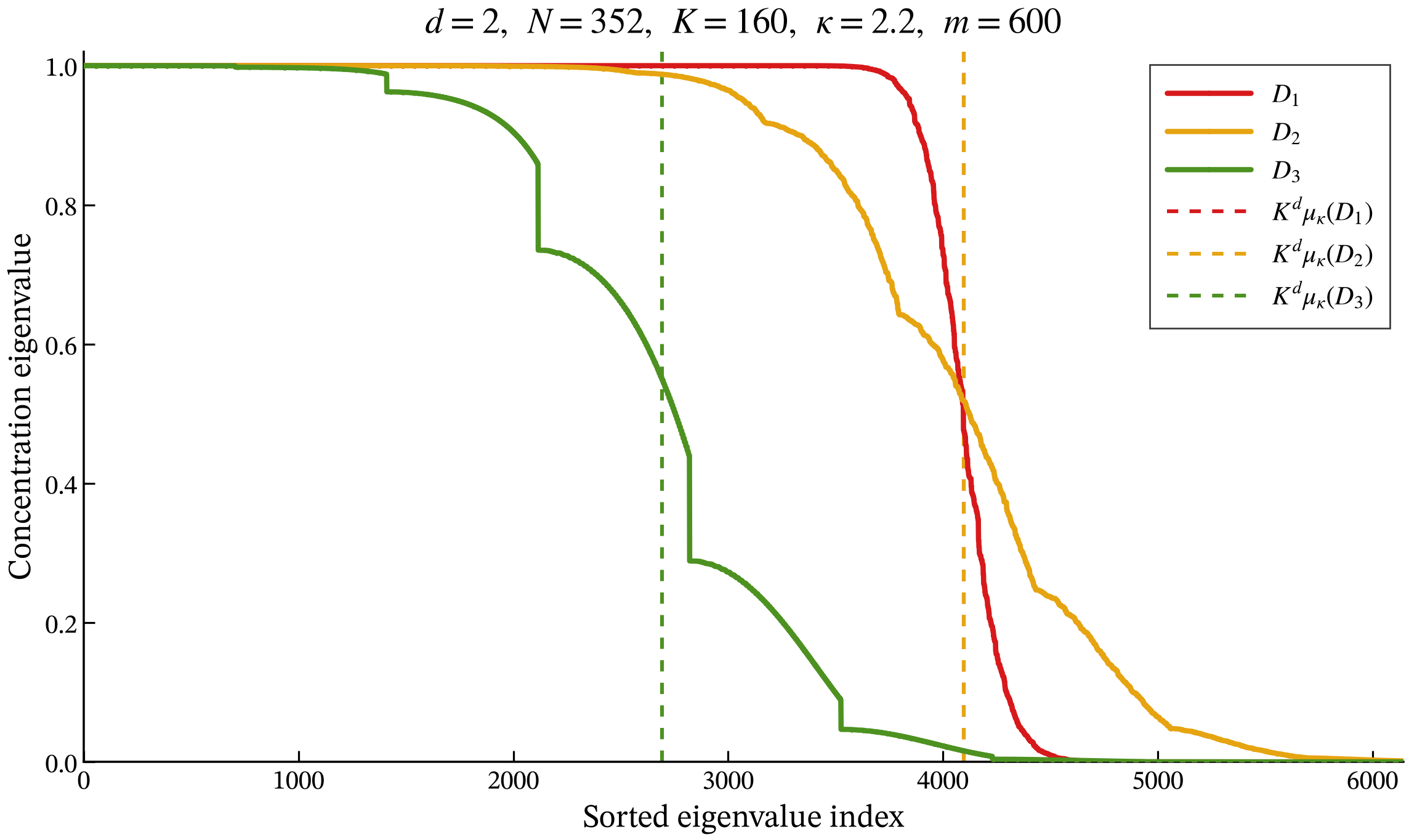}
\caption{$\kappa=2.2$ ($K=160$, $\degAInCFive=352$). The predicted indices $K^2\mu_\kappa(D_i)$ are $4096.000$, $4096.000$, and $2692.790$.}
\label{fig:eigen-spectra-d2-kappa22}
\end{subfigure}
\caption{Sorted SFB concentration eigenvalues for three equal-area radial domains in dimension $d=2$. Here $D_i=\{x\in\R^2:\|x\|\in\mathcal I_i\}$, with $\mathcal I_1=[0,0.8]$, $\mathcal I_2\approx[2,2.15407]$, and $\mathcal I_3\approx[4,4.07922]$. The dashed vertical lines mark the theoretical leading counts $K^2\mu_\kappa(D_i)$ calculated from~\eqref{eq:app-radial-mu}. In panel~(b), the red and yellow dashed lines coincide.}
\label{fig:eigen-spectra-d2-comparison}
\end{figure}
\FloatBarrier

Figure~\ref{fig:eigen-spectra-d2-comparison} displays two distinct consequences of the radial Shannon coefficient. First, for $\kappa=1$, the three domains have identical Euclidean area but substantially different predicted plunge locations. The disk $D_1$ lies in the constant Paley--Wiener-like plateau and has predicted count $4096.000$. The annuli $D_2$ and $D_3$ sample progressively smaller values of the transition profile and have predicted counts $2410.039$ and $1277.703$, respectively. The observed spectral transitions occur near the corresponding dashed lines. Thus equal area does not imply equal leading effective dimension once the angular cutoff has broken translation invariance.

Second, comparison of the two panels shows that radius alone is not the relevant variable; the profile is sampled at the scaled radius $r/\kappa$. Increasing $\kappa$ from $1$ to $2.2$ expands the constant-density region. Consequently, both $D_1$ and $D_2$ lie in the plateau for $\kappa=2.2$, and their predicted counts coincide:
\[
K^2\mu_\kappa(D_1)=K^2\mu_\kappa(D_2)=4096.000.
\]
Their spectra plunge at essentially the same index even though the domains occupy very different radial locations. The outer annulus $D_3$ still reaches the transition region, so its smaller predicted count $2692.790$ produces an earlier plunge.

The dashed lines are determined entirely by the theoretical quantity $K^2\mu_\kappa(D_i)$ and are not adjusted using the eigenvalue curves. Their agreement with the observed transition indices therefore illustrates two central consequences of Theorem~\ref{thm:eigendistri}: radial placement matters outside the plateau, and increasing the angular bandwidth pushes the Paley--Wiener-like regime farther from the origin.

\section{Concluding remarks}

We have studied a two-parameter family of spectral projections obtained from the classical Paley--Wiener projection by imposing an additional spherical-harmonic cutoff. This cutoff preserves rotations but destroys translation invariance, and this loss of translation invariance is reflected in the diagonal reproducing kernel. In the coupled regime \(N/K\to\kappa\), the normalized diagonal kernel \(K^{-d}K_{N,K}(x,x)\) converges to a nonconstant radial local density, which interpolates between the constant Paley--Wiener density near the origin and a Hankel-type radial-density law in the far field.

This local-density point of view determines the leading spectral behavior of the associated concentration operators. The Shannon number is no longer determined only by the Euclidean volume of the localization domain, but by the integral of the limiting radial density over that domain. Thus domains with the same volume may have different leading effective dimensions when they occupy different radial regions. In this sense, the usual space-bandwidth principle is modified by the non-translation-invariant angular cutoff.

These results also give a quantitative interpretation of the spatial resolution induced by SFB truncations. In practical spherical-coordinate representations, the radial bandwidth \(K\) and the angular cutoff \(N\) are often chosen separately. The present analysis shows that their ratio \(\kappa=N/K\) determines the radial scale at which the Paley--Wiener-like plateau gives way to the Hankel-type far-field regime. Varying this truncation ratio therefore changes the effective spatial resolution across radial locations. This provides a way to interpret how radial bandwidth, angular resolution, and radial placement interact in SFB-based localized reconstruction problems, and suggests that the choice of \(N\) relative to \(K\) can be used to balance resolution across different spatial regions.


\printbibliography

\appendix

\section{Auxiliary proofs}\label{secC5:appendix}

\begin{prfof}{~\eqref{eqnC5:analysiDimd_2}} Assume that $0<r\leq1$; the case $r>1$ is similar and simpler. By~\eqref{eqnC5:analysiDimd_1},
	\begin{align}
		&\quad \frac{\sum_{\degaInCFive=0}^{\degAInCFive} {C}^{\langle d \rangle}_{\degaInCFive} J^2_{\degaInCFive+\frac{d-2}{2}}(\degAInCFive r)}{{C}^{\langle d \rangle}_{\degAInCFive}}-\left[ \functionInCFiveFA(r)- \sum_{\degaInCFive =0}^{\degAInCFive-1}\frac{{C}^{\langle d \rangle}_{\degaInCFive+1}-{C}^{\langle d \rangle}_{\degaInCFive}}{{C}^{\langle d \rangle}_{\degAInCFive}}\functionInCFiveFA\left(\tfrac{\degAInCFive}{\degaInCFive}r\right)\right]
		\\
		\nonumber &=\underbrace{\left[\sum_{\degaInCFive=0}^{\degAInCFive}J^2_{\degaInCFive+\frac{d-2}{2}}(\degAInCFive r)-\functionInCFiveFA(r)\right]}_{=:S1(\degAInCFive)}+\underbrace{\sum_{\degaInCFive =0}^{\degAInCFive-1}\frac{{C}^{\langle d \rangle}_{\degaInCFive+1}-{C}^{\langle d \rangle}_{\degaInCFive}}{{C}^{\langle d \rangle}_{\degAInCFive}}\left[\functionInCFiveFA\left(\tfrac{\degAInCFive}{\degaInCFive}r\right)-\sum_{{\degaInCFive}^{'}=0}^{\degaInCFive} J^2_{{\degaInCFive}^{'}+\frac{d-2}{2}}(\degAInCFive r)\right]}_{=:S2(\degAInCFive)}.
	\end{align}
	We abbreviate $\frac{{C}^{\langle d \rangle}_{\degaInCFive+1}-{C}^{\langle d \rangle}_{\degaInCFive}}{{C}^{\langle d \rangle}_{\degAInCFive}}$ by $a_{\degaInCFive}$.
	By~\eqref{eqnC5:keyequation}, $S1(\degAInCFive)\to0$ as $\degAInCFive\to\infty$. Treating $S2(\degAInCFive)$ requires uniform convergence, so we split the sum into four parts as in Remark~\ref{remC5:remarkofKeyequation}. For any $\epsilon>0$, choose $\delta>0$ such that
	$$\tfrac{3}{2}\{\delta^{d-2}+[1-(1-\delta)^{d-2}]r^{d-2}\}<\epsilon.$$
	Using this $\delta$, divide $I=\{0,\ldots,\degAInCFive-1\}$ into four subsets:
	\begin{align*}
		I_1=I\cap(0,\delta \degAInCFive)\quad I_2=I\cap[\delta \degAInCFive,(1-\delta)r\degAInCFive)\quad I_3=I\cap[(1-\delta)r\degAInCFive,r\degAInCFive)\quad I_4=I\cap[r\degAInCFive,\degAInCFive].
	\end{align*} Then
	\begin{align}
		S2(\degAInCFive)=\sum_{i=1,2,3,4}S2_i(\degAInCFive)\quad \text{where}\quad S2_i(\degAInCFive)=\sum_{\degaInCFive\in I_i}a_{\degaInCFive}\left[\functionInCFiveFA\left(\tfrac{\degAInCFive}{\degaInCFive}r\right)-\sum_{{\degaInCFive}^{'}=0}^{\degaInCFive} J^2_{{\degaInCFive}^{'}+\frac{d-2}{2}}(\degAInCFive r)\right]
	\end{align}
	Note that $\sum_{\degaInCFive\in I}a_{\degaInCFive}\leq1$. For $i=2,4$, we have $\frac{\degAInCFive r}{\degaInCFive}\in(\frac{1}{1-\delta},\frac{r}{\delta}]$ and $\frac{\degAInCFive r}{\degaInCFive}\in[r,1]$, respectively. By the uniform convergence in~\eqref{eqnC5:keyequation} on these regions (Remark~\ref{remC5:remarkofKeyequation}), $S2_2(\degAInCFive)$ and $S2_4(\degAInCFive)$ tend to zero. For $i=1,3$, we use $\functionInCFiveFA(s)\leq\tfrac12$ and the uniform bound $\sum_{{\degaInCFive}^{'}=0}^{\degaInCFive}J^2_{{\degaInCFive}^{'}+\frac{d-2}{2}}(\degAInCFive r)\leq1$. Thus
	\begin{align}
		\nonumber	|S2_1(\degAInCFive)+S2_3(\degAInCFive)| \leq \tfrac{3}{2}\left(\sum_{\degaInCFive\in I_1\cup I_3}a_{\degaInCFive}\right) &=\tfrac{3}{2}\{\delta^{d-2}+r^{d-2}(1-(1-\delta)^{d-2})+\bigo({\degAInCFive}^{-1})\}\\
		&\leq \epsilon +\bigo(\degAInCFive^{-1})
	\end{align}
	where we used~\eqref{eqnC5:asympoticofdCl} to estimate $\sum_{\degaInCFive\in I_1\cup I_3}a_{\degaInCFive}$. 
	This yields $$\limsup_{\degAInCFive\to\infty} |S2(\degAInCFive)| \leq \sum_{i=1,2,3,4}\limsup_{\degAInCFive\to\infty} |S2_i(\degAInCFive)|<\epsilon.$$
	Since $\epsilon$ is arbitrary, $S2(\degAInCFive)\to 0$, which completes the proof of~\eqref{eqnC5:analysiDimd_2}.
\end{prfof}

\begin{prfof}{~\eqref{eqnC5:summation2integral}} Let $r>0$ be fixed. We define two functionals on $C(\R_{+})$:
\begin{align}
\opFont{A}_{\degAInCFive}(f):=\sum_{\degaInCFive =0}^{\degAInCFive-1} \frac{{C}^{\langle d \rangle}_{\degaInCFive+1}-{C}^{\langle d \rangle}_{\degaInCFive}}{{C}^{\langle d \rangle}_{\degAInCFive}}f(\frac{\degAInCFive}{\degaInCFive}r), &&\text{and}&& \opFont{B}(f):=\int_{1}^{\infty}f(tr) \left[(d-2)t^{1-d}\right] \diffsymbol t.
\end{align}
Both $\opFont{A}_{\degAInCFive}$ and $\opFont{B}$ are linear and positive in the following sense: if $f(t)\geq g(t)$ for every $t$, then $\opFont{A}_{\degAInCFive}(f)\geq \opFont{A}_{\degAInCFive}(g)$ and $\opFont{B}(f)\geq \opFont{B}(g)$.

First, consider $f=\chi_{(0,x]}$ and $r=1$. In this case, the difference between $\opFont{A}_{\degAInCFive}(f)$ and $\opFont{B}(f)$ depends only on $x$ and $\lceil \degAInCFive/x\rceil$. In particular,
\begin{align}
|\opFont{A}_{\degAInCFive}(f)- \opFont{B}(f)|\leq x^{2-d}-\frac{{C}^{\langle d \rangle}_{\lceil \degAInCFive/x\rceil}}{{C}^{\langle d \rangle}_{\degAInCFive}}=\bigo(\frac{(\degAInCFive/x)^{d-3}}{\degAInCFive^{d-2}})=\bigo(\frac{1}{\degAInCFive x^{d-3}}),
\end{align}
which yields $\lim_{\degAInCFive\to\infty}\opFont{A}_{\degAInCFive}(f)=\opFont{B}(f)$. If $r\neq 1$, the same convergence follows by considering $f_{\langle r \rangle}=\chi_{(0,x/r]}$. We also allow $x=\infty$, corresponding to $f=\chi_{(0,\infty)}$. In this case, the definition gives $|\opFont{A}_{\degAInCFive}(f)-\opFont{B}(f)|=\tfrac{1}{{C}^{\langle d \rangle}_{\degAInCFive}}=\bigo(\degAInCFive^{2-d})$.

The function $\functionInCFiveFA$ defined in Definition~\ref{def:expressiong} is continuous and decreases to zero. Hence, for every $\epsilon>0$, there exist functions $F_1,F_2$ of the form
\begin{align}
F_{i}=\sum_{n=1}^{N_{i}} c_{n}^{i}\chi_{(0,x_{m}^{i}]} \quad\text{for}\quad i\in\{1,2\}, c_{m}^{i}\in \R ,x_{m}^{i}\in \R\cup\{\infty\},
\end{align}
such that 
\begin{align}
F_1\leq \functionInCFiveFA\leq F_2 \quad\text{and }\quad \max_{i\in\{1,2\}}\{ |\opFont{B}(\functionInCFiveFA)-\opFont{B}(F_i)|\}\leq \epsilon.
\end{align}
By the previous discussion and the linearity of $\opFont{A}_{\degAInCFive}$ and $\opFont{B}$, we have $\lim_{\degAInCFive\to\infty}\opFont{A}_{\degAInCFive}(F_i)=\opFont{B}(F_i)$ for $i\in\{1,2\}$. Using the positivity of the two functionals, we obtain
\begin{align}
\opFont{B}(\functionInCFiveFA)-\epsilon\leq \lim_{\degAInCFive\to\infty}\opFont{A}_{\degAInCFive}(F_1)& \leq \liminf_{\degAInCFive\to\infty} \opFont{A}_{\degAInCFive}(\functionInCFiveFA)\\ \nonumber
& \leq \limsup_{\degAInCFive\to\infty}\opFont{A}_{\degAInCFive}(\functionInCFiveFA) \leq \lim_{\degAInCFive\to\infty}\opFont{A}_{\degAInCFive}(F_2)\leq \opFont{B}(\functionInCFiveFA)+\epsilon.
\end{align}
Since $\epsilon$ is arbitrary, $\opFont{A}_{\degAInCFive}(\functionInCFiveFA)$ is a Cauchy sequence and converges to $\opFont{B}(\functionInCFiveFA)$.
\end{prfof}

\section{Additional numerical observations}\label{secC5:numericalremarks}
This appendix collects supplementary numerical observations that are not used in the proofs of the main results. Subsection~\ref{subsec:bounded-ball-observations} concerns SFB systems on the unit ball and should be viewed as a bounded-domain analogue of the SFB truncation model on $\R^d$. Subsection~\ref{subsec:fullspace-neardiagonal} returns to the reproducing kernel $\KLK$ on $\R^d$ and records additional near-diagonal kernel observations.
\subsection{Spherical Fourier-Bessel series on the unit ball}\label{subsec:bounded-ball-observations}

With appropriate boundary conditions, SFB series can be used to construct orthogonal bases for the bounded ball $\BBCFour{r}=\{x:\|x\|\leq r\}\subset \R^d$ (see also \cite{WRB08} for a discussion of formulations and applications of SFB on $\BB^d$ with $d=2,3$). For convenience, we consider only the unit ball, i.e., $\BBCFour{1}=\BB^d$. Two typical choices are the Dirichlet and Neumann boundary conditions, which respectively transform~\eqref{eqnC5:Helmholtz equation} into
\begin{align*}
\begin{cases}
(	\laplacian +k^2)f=0,\\
f|_{ \Sphere} =0.
\end{cases}
&&
\begin{cases}
(	\laplacian +k^2)f=0,\\
\frac{\partial}{\partial r}f|_{ \Sphere} =0.
\end{cases}
\end{align*}
In both cases, the spherical and radial variables can still be separated as in~\eqref{eqnC5:sp-helmholtz} and~\eqref{eqnC5:r-helmholtz}, with the additional boundary condition for $R(r)$ given by
\begin{align*}
R(1) =0
&&\text{or}&&
\frac{\diffsymbol}{\diffsymbol r}R(1) =0.	
\end{align*}
The separated solutions retain the form
\begin{align*}
f(r\xi)={C} r^{\frac{2-d}{2}}J_{\degaInCFive+\frac{d-2}{2}}(kr)Y_{\degaInCFive,\degb}(\xi),\qquad {C}\in\R.
\end{align*}
In this setting, however, the boundary condition restricts $k$ to a discrete set. For each $\degaInCFive\in\N_0$, the admissible values are the positive roots of
\begin{align}
J_{\degaInCFive+\frac{d-2}{2}}(k)=0,
\end{align}
for the Dirichlet condition, and the positive roots of
\begin{align}
\degaInCFive J_{\degaInCFive+\frac{d-2}{2}}(\tilde{k})-\tilde{k}J_{\degaInCFive+\frac{d}{2}}(\tilde{k})=0,
\end{align} 
for the Neumann boundary condition. In the Neumann case, the exceptional value $\tilde{k}=0$ occurs when $\degaInCFive=0$ and corresponds to constant functions. We denote the respective roots by $k_{\degaInCFive_i}$ and $\tilde{k}_{\degaInCFive_i}$. By Sturm--Liouville theory, for each $\degaInCFive$ there are countably many simple roots, which may be arranged in increasing sequences
\begin{align}
k_{\degaInCFive_1}< k_{\degaInCFive_2}<\cdots <k_{\degaInCFive_{i}}<\cdots,
\end{align}
and similarly for $\tilde{k}$. We now define the space $\spaceInCFive_{\degAInCFive, K}$ on $\BB^d$ by
\begin{align}\label{def:SHBseriesWithDirichletBoundary}
\spaceInCFive_{\degAInCFive, K}(\BB^d):=\mathrm{span}\{r^{\frac{2-d}{2}}J_{\degaInCFive+\frac{d-2}{2}}(k_{\degaInCFive_{i}}r) Y_{\degaInCFive,\degb}(\xi): k_{\degaInCFive_{i}}\leq K,  \degaInCFive \leq \degAInCFive,\degb\leq\mathrm{dim}(H_{\degaInCFive}^{d})\}.
\end{align}
This defines the SFB space with the Dirichlet boundary condition. The Neumann space is defined analogously by replacing $k_{\degaInCFive,i}$ with $\tilde{k}_{\degaInCFive,i}$ in~\eqref{def:SHBseriesWithDirichletBoundary}.
Write
\begin{align*}
p_{i,\degaInCFive,\degb}:={C}_{i,\degaInCFive,\degb} r^{\frac{2-d}{2}}J_{\degaInCFive+\frac{d-2}{2}}(k_{\degaInCFive_{i}}r) Y_{\degaInCFive,\degb}(\xi),
\end{align*}
where ${C}_{i,\degaInCFive,\degb}$ denotes the normalization constant such that $\int_{\BB^d} |p_{i,\degaInCFive,\degb}(x)|^2\diffsymbol x =1$. Then the reproducing kernel for $\spaceInCFive_{\degAInCFive, K}(\BB^d)$ is
\begin{align}\label{def:RPKforSHBseriesWithDirichletBoundary}
\KLK(\BB^d;x,y)= \sum_{k_{\degaInCFive_{i}}\leq K}\sum_{\degaInCFive \leq \degAInCFive}\sum_{\degb\leq \mathrm{dim}(H_{\degaInCFive}^{d})} p_{i,\degaInCFive,\degb}(x) \overline{p_{i,\degaInCFive,\degb}(y)}. 
\end{align}
Numerically, we observe that the diagonal reproducing kernel $\KLK(\BB^d;x,x)$, associated with either the Dirichlet or Neumann boundary condition and with $\degAInCFive=\kappa K$, also seems to be asymptotically characterized by $K^d\times {\functionInCFiveFB}^{\langle d \rangle}_{\langle \kappa \rangle}$. Figure \ref{fig:SHBseriesOnUnitBall} shows $\KLK(\BB^d;x,x)$ normalized by $K^d$, with the Dirichlet boundary condition in the left column, the Neumann boundary condition in the right column, $d=2$ in the top row, and $d=3$ in the bottom row. The boundary conditions strongly influence the behavior of $\KLK(\BB^d;x,x)$ near $\Sphered$: the Dirichlet condition requires $p_{i,\degaInCFive,\degb}(x)=0$ for $\|x\|=1$, whereas the Neumann condition forces $|p_{i,\degaInCFive,\degb}(x)|$ to attain a directional local maximum at the boundary. In the interior of the ball, however, $\frac{\KLK(\BB^d;x,x)}{K^d}$ shows a clear trend of convergence to ${\functionInCFiveFB}^{\langle d \rangle}_{\langle \kappa \rangle}$.

\begin{figure}[!htbp]
\centering
\begin{subfigure}[t]{0.48\textwidth}
\centering
\includegraphics[width=\linewidth]{Fig_B_Dim2D_bc}
\caption{$d=2$, Dirichlet boundary condition.}
\end{subfigure}
\hfill
\begin{subfigure}[t]{0.48\textwidth}
\centering
\includegraphics[width=\linewidth]{Fig_B_Dim2N_bc}
\caption{$d=2$, Neumann boundary condition.}
\end{subfigure}

\medskip
\begin{subfigure}[t]{0.48\textwidth}
\centering
\includegraphics[width=\linewidth]{Fig_B_Dim3D_bc}
\caption{$d=3$, Dirichlet boundary condition.}
\end{subfigure}
\hfill
\begin{subfigure}[t]{0.48\textwidth}
\centering
\includegraphics[width=\linewidth]{Fig_B_Dim3N_bc}
\caption{$d=3$, Neumann boundary condition.}
\end{subfigure}

\caption{Normalized diagonal reproducing kernels $\mathcal{K}_{\degAInCFive,K}(\B^d;x,x)/K^d$ for SFB systems on the unit ball with $\kappa=\degAInCFive/K=0.5$. The left and right columns correspond to the Dirichlet and Neumann boundary conditions, respectively, while the top and bottom rows correspond to $d=2$ and $d=3$. The dotted curves show the corresponding limiting profiles ${\functionInCFiveFB}^{\langle d\rangle}_{\langle\kappa\rangle}(\|x\|)$ for the SFB truncation spaces on $\R^d$.}
\label{fig:SHBseriesOnUnitBall}
\end{figure}
\FloatBarrier

\subsection{\texorpdfstring{Supplementary near-diagonal kernel observations on $\R^d$}{Supplementary near-diagonal kernel observations on R d}}\label{subsec:fullspace-neardiagonal}
Another quantity of interest is the ratio $\frac{\KLK(x,x+\nicefrac{y}{K})}{\KLK(x,x)}$. Its behavior may provide a more precise description of the concentration of $\KLK(x,y)$ near the diagonal $x=y$, and thereby sharpen Proposition~\ref{prop:concentrationofK\degAInCFive, K}. Numerical experiments suggest that, under the relation $\degAInCFive=\kappa K$, the quantity $\lim_{K\to\infty}\frac{\KLK(x,x+\nicefrac{y}{K})}{\KLK(x,x)}$ converges to a function 
that depends on $\|\kappa^{-1}x\|$, $\|y\|$, and the angle $\theta$ between $x$ and $y$ 
($\theta=\arccos(\frac{\langle x,y\rangle}{\|x\|\|y\|})$ for $x,y\neq 0$).

In particular, if $\|\kappa^{-1}x\|\leq 1$, then the limit function is rotationally invariant in $y$ (i.e., independent of $\theta$) and has the form
\begin{align}\label{eqnC5:state1}
\lim_{\degAInCFive/K=\kappa;\,K\to \infty}\frac{\KLK(x,x+\nicefrac{y}{K})}{\KLK(x,x)}	
= 2^{d/2}\substigamma(d/2+1)\frac{J_{d/2}(\|y\|)}{\|y\|^{d/2}},
\end{align}
where the value at $y=0$ is defined by continuity. This is the normalized local form of the reproducing kernel of $\mathrm{PW}_{K}$. It can again be interpreted through the phase delay of Bessel functions: the angular-momentum truncation has almost no influence when $\|x\|\leq \kappa$. However, as $\|x\|$ grows, the limit function gradually deforms and rotation invariance disappears. Eventually, when $\kappa^{-1}\|x\|$ becomes large, the limit function appears to have the following approximate multiplicative structure:
\begin{align}\label{eqnC5:state2}
\lim_{\degAInCFive/K=\kappa;\,K\to \infty}\frac{\KLK(x,x+\nicefrac{y}{K})}{\KLK(x,x)}
\approx \mathrm{sinc}(\|y\|\cos\theta)\,
2^{\frac{d-1}{2}}\substigamma\left(\frac{d+1}{2}\right)
\frac{J_{\frac{d-1}{2}}(z)}{z^{\frac{d-1}{2}}},
\end{align}
where $z=\frac{\|y\|\sin\theta}{\kappa^{-1}\|x\|}$, both Bessel quotients are defined by continuity at zero, and
\[
\mathrm{sinc}(t)=\frac{\sin t}{t}
=2^{1/2}\substigamma(3/2)\frac{J_{1/2}(t)}{t^{1/2}}.
\]
The quantities $\|y\|\cos\theta$ and $\|y\|\sin\theta$ are the lengths of the components of $y$ parallel and perpendicular to $x$, respectively. These statements are included only as supplementary numerical observations and are not used elsewhere in the paper. At present we do not have a proof of the existence of the limit function, nor a satisfactory characterization of the deformation from~\eqref{eqnC5:state1} to~\eqref{eqnC5:state2} as $\|x\|$ increases.

\end{document}